\documentclass{article}

\usepackage{mystyle}
\usepackage{wasysym}
\begin{document}
\title{Some combinatorial properties of Ultimate \(L\) and \(V\)}
\author{Gabriel Goldberg\\ Evans Hall\\ University Drive \\ Berkeley, CA 94720}
\maketitle
\begin{abstract}
    This paper establishes a number of constraints on the 
    structure of large cardinals under strong compactness assumptions. 
    These constraints coincide with those imposed by the Ultrapower
    Axiom \cite{UA}, a principle that is expected to hold in Woodin's hypothesized Ultimate \(L\),
    providing some evidence for the Ultimate \(L\) Conjecture \cite{Woodin}.

    We show that every regular cardinal above the first strongly compact that carries
    an indecomposable ultrafilter is measurable, 
    answering a question of Silver \cite{Silver} for large enough cardinals. 
    We show that any successor almost strongly compact
    cardinal of uncountable cofinality is strongly compact,
    making progress on a question
    of Boney, Unger, and Brooke-Taylor \cite{Boney}.
    We show that if there is a proper class of strongly
    compact cardinals then there is no nontrivial cardinal preserving elementary embedding
    from the universe of sets into an inner model,
    answering a question of Caicedo \cite{Caicedo} granting large cardinals. 
    Finally, we show that
    if \(\kappa\) is strongly compact, then \(V\) is a set forcing extension 
    of the inner model \(\kappa\text{-HOD}\) consisting of sets that are hereditarily ordinal definable from a
    \(\kappa\)-complete ultrafilter over an ordinal; \(\kappa\text{-HOD}\) 
    seems to be the first nontrivial
    example of a ground of \(V\) whose definition does not involve forcing.
\end{abstract}
\section{Introduction}
\subsection{The Ultimate \(L\) Conjecture}
Since Cohen's proof of the independence of the Continuum Hypothesis \cite{Cohen},
it has become clear that many of the fundamental features of the universe of sets
will never be decided on the basis of the currently accepted axioms of set theory.
Woodin's Ultimate \(L\) Conjecture
\cite{Woodin}, 
however, raises the possibility that the fundamental objects of set theory
can be transferred into a substructure of the universe (namely, Ultimate \(L\))
whose theory is as tractable as those of the conventional structures of mathematics.\footnote{
    {\it The axiom \(V = \text{Ultimate }L\)}:\ (1)\ There is a proper class of Woodin cardinals.\ 
    (2)\ If some level of the von Neumann hierarchy 
    satisfies a sentence \(\varphi\) in the language of set theory, 
    then there is a universally Baire set \(A\subseteq \mathbb R\) 
    such that some level of the von Neumann hierarchy of 
    \(\textnormal{HOD}^{L(A,\mathbb R)}\) satisfies \(\varphi\).

    {\it The Ultimate \(L\) Conjecture}:
If \(\kappa\) is extendible, then
there is an inner model \(M\) that satisfies ZFC plus the axiom \(V = \text{Ultimate }L\)
and has the property that for all cardinals \(\lambda\geq \kappa\),
there is a \(\kappa\)-complete normal fine ultrafilter \(\mathcal U\) over 
\(P_\kappa(\lambda)\) with \(P_\kappa(\lambda)\cap M\in \mathcal U\) and \(\mathcal U\cap M\in M\).
}
The fundamental objects in question
are {\it large cardinals},
strong closure points in Cantor's hierarchy of infinities
whose existence, taken axiomatically,
suffices to interpret and compare the vast array of 
mutually incompatible formal systems studied in contemporary set theory.

If Woodin's conjecture is true, the downward transference of large cardinal properties
from the universe of sets into Ultimate \(L\) would
necessitate an upward transference of 
combinatorial structure from Ultimate \(L\) back into the universe of sets. 
(For example, see \cite[Theorem 8.4.40]{UA}.)
As a consequence, the true believer should predict that the universe of sets
resembles Ultimate \(L\) in certain ways.
This paper presents a collection of theorems confirming this prediction
by showing that various consequences of the Ultrapower Axiom, 
a principle expected to hold in Ultimate \(L\), are actually provable 
from large cardinal axioms alone.

\subsection{The Ultrapower Axiom}
The Ultrapower Axiom (UA)
asserts that the category of wellfounded ultrapowers of the
universe of sets with internally definable ultrapower embeddings
is directed.\footnote{A category theorist would say {\it filtered}.}
In the author's thesis \cite[Theorem 2.3.10]{UA}, it is shown
that UA holds in any model whose
countable elementary substructures satisfy
a weak form of the {\it Comparison Lemma} of inner model theory.

The Comparison Lemma is really a
series of results (for example, \cite{Godel, KunenLU, MitchellSteel, NeemanSteel, FiniteLevels})
each roughly asserting the directedness of 
some subcategory of the category of (countable) canonical models of set theory
with iterated ultrapower embeddings. These canonical models are also known as ``mice.'' We
warn that the ``category of canonical models'' is not yet precisely defined;
so far, only certain subcategories of this category have been identified,
namely, those for which the Comparison Lemma has been proved.
The term ``iterated ultrapower'' is used in a similarly open-ended sense.

As it is currently conceived, 
the ongoing search for more powerful canonical models of set theory
(including Ultimate \(L\)) amounts to an
attempt to generalize the Comparison Lemma to larger subcategories of
the category of canonical models.
As a consequence, the current methodology of inner model theory
simply cannot produce a canonical model in which the Ultrapower 
Axiom fails. For this reason, it seems likely that if
Ultimate \(L\) exists, it will satisfy the Ultrapower Axiom.
\subsection{Consequences of UA from large cardinal axioms alone}
The Ultrapower Axiom can be used to develop a structure theory
in the context of very large cardinals, proving, for example,
that the Generalized Continuum Hypothesis holds
above the least strongly compact cardinal and
that the universe is a set generic extension of \(\textnormal{HOD}\).
One can also develop the theory of large cardinals,
obtaining equivalences between a number of large cardinal axioms
that are widely believed to have the same strength
(e.g., strong compactness and supercompactness).

All of these results are impossible to prove 
in ZFC alone, but it turns out that each has an analog
that is provable from large cardinal axioms.
For example, the analog of the UA theorem that the GCH
holds above a strongly compact cardinal is
Solovay's result that the Singular Cardinals Hypothesis holds above a strongly compact cardinal.
This paper establishes analogs of the other theorems 
using techniques that are quite different from those used under the Ultrapower Axiom.
The main methods of this paper actually derive from a lemma used by Woodin in
his analysis of the downward transference of large cardinal
axioms to Ultimate \(L\), namely, that
assuming large cardinal axioms, any ultrapower of the universe
absorbs all sufficiently complete ultrafilters (\cref{AmenableThm}). This
fact enables us to simulate the Ultrapower Axiom in
certain restricted situations.

We now summarize the results of this paper.
\subsection{Indecomposable ultrafilters and Silver's question}
Our first theorem, the subject of \cref{SilverSection}, 
concerns a question posed by Silver \cite{Silver} in the 1970s.
If \(\delta \leq \lambda\) are cardinals, \(X\) is a set, and \(U\) is an ultrafilter over
\(X\), \(U\) is said to be {\it \((\delta,\lambda)\)-indecomposable} if any
partition \(\langle A_\nu\rangle_{\nu < \alpha}\) of \(X\) with \(\alpha < \lambda\)
has a subsequence \(\langle A_{\nu_\xi}\rangle_{\xi < \beta}\) with \(\beta < \delta\)
such that \(\bigcup_{\xi < \beta} A_{\nu_\xi}\in U\).
Indecomposability refines the concept of \(\lambda\)-completeness: an
ultrafilter \(U\) over \(X\) is {\it \(\lambda\)-complete} if 
whenever \(\langle A_\nu\rangle_{\nu < \alpha}\) is a partition of \(X\) with \(\alpha < \lambda\),
there is some \(\nu < \alpha\) such that \(A_\nu\in U\), or in other words, if
\(U\) is \((2,\lambda)\)-indecomposable (equivalently, \((\omega,\lambda)\)-indecomposable). 

The precise relationship between indecomposability and completeness, however, is not at all clear. 
A uniform ultrafilter on a cardinal \(\lambda\) is said to be {\it indecomposable} if it
is \((\omega_1,\lambda)\)-indecomposable, the maximum degree of 
indecomposabiliy short of \(\lambda\)-com\-pleteness.
Silver asked whether an inaccessible cardinal \(\lambda\) that carries
an indecomposable ultrafilter is necessarily
measurable. The underlying idea is that such an ultrafilter is very close to being
\(\lambda\)-complete and hence witnessing the measurability of \(\lambda\).\footnote{
    We caution that assuming mild large cardinal axioms
    there {\it is} a uniform \((\omega_1,\lambda)\)-indecomposable ultrafilter over an
    inaccessible cardinal that is not itself \(\lambda\)-complete.
    Silver's question is really whether one can extract a \(\lambda\)-complete ultrafilter
    from any such ultrafilter.}
\begin{comment}
If there are infinitely many measurable cardinals \(\langle \kappa_n\rangle_{n <
\omega}\), there is an indecomposable ultrafilter over \(\lambda = \sup_{n <
\omega} \kappa_n\) that is not \(\lambda\)-complete: if  \(D\) is a nonprincipal
ultrafilter over \(\omega\) and for each \(n < \omega\), \(U_n\) is a
nonprincipal \(\kappa_n\)-complete ultrafilter over \(\kappa_n\), then the
\(D\)-limit of the ultrafilters \((U_n)_{n < \omega}\), which consists of those
sets \(A\subseteq \lambda\) with \(A\cap \kappa_n \in U_n\) for \(D\)-almost all
\(n < \omega\), is an indecomposable ultrafilter over \(\lambda\).
\end{comment}

Jensen showed that in the canonical inner models, the answer to Silver's
question is yes. On the other hand, by forcing, Sheard \cite{Sheard} produced a
model in which the answer is no. Thus the question appears to be ``settled'' in
the usual way: no answer can be derived from the standard axioms. 

The Ultrapower Axiom does not help with Silver's question itself, but it does
answer the natural generalization of Silver's question to countably complete
ultrafilters: assuming UA, if \(\delta\) is a cardinal and \(\lambda > \delta\)
is inaccessible and carries a uniform countably complete
\((\delta,\lambda)\)-indecomposable ultrafilter, then \(\lambda\) is
measurable.

Despite Jensen and Sheard's independence results, we will show that for
sufficiently large cardinals \(\lambda\), the answer to Silver's question is
yes:
\begin{repthm}{IndecompChar0}
Suppose \(\delta < \kappa\leq \lambda\) are cardinals, \(\kappa\) is strongly
compact, and \(\lambda\) carries a uniform \((\delta, \lambda)\)-indecomposable
ultrafilter. Then \(\lambda\) either is a measurable cardinal or 
\(\lambda\) has cofinality less than \(\delta\) and \(\lambda\) is a limit of measurable cardinals.
\end{repthm}
As a consequence, if a cardinal \(\lambda\) above the least strongly compact
cardinal carries an indecomposable ultrafilter, then \(\lambda\) is either
a measurable cardinal or the limit of countably many measurable cardinals. 
\begin{comment}
In fact, one can completely characterize the indecomposable ultrafilters above
the least strongly compact cardinal as exactly those ultrafilters resulting from
the construction above.
\begin{repthm}{IndecompChar1}
Suppose \(\kappa\) is strongly compact, \(\lambda\geq \kappa\), and \(U\) is an
indecomposable ultrafilter over \(\lambda\). Then \(U\) is either
\(\lambda\)-complete or else for some ultrafilter \(D\) on \(\omega\), some
sequence \(\langle\kappa_n\rangle_{n < \omega}\) of measurable cardinals, and
some sequence \(\langle U_n\rangle_{n < \omega}\) of \(\kappa_n\)-complete
ultrafilters on \(\kappa_n\), \(U\) is the  \(D\)-limit of the ultrafilters
\(\langle U_n\rangle_{n < \omega}\).
\end{repthm}
\end{comment}

\subsection{Almost strong compactness}
Our second result, proved in \cref{AlmostSection}, 
concerns a generalization of strong compactness defined by
Bagaria-Magidor \cite{MagidorBagaria}. A cardinal \(\kappa\) is 
{\it strongly compact} if every \(\kappa\)-complete filter extends to a \(\kappa\)-complete
ultrafilter. Many applications of strong compactness only seem to require that \(\kappa\) be
{\it almost strongly compact}: for any cardinal \(\nu < \kappa\), every \(\kappa\)-complete filter
extends to a \(\nu\)-complete ultrafilter.

The Ultrapower Axiom's most interesting consequences relate to the structure of strong compactness. 
Most notably, UA implies that the least strongly compact cardinal is supercompact.
In fact, UA also implies that the least {\it almost} strongly compact cardinal is supercompact;
in particular, the least almost strongly compact cardinal is strongly compact. 
Whether this is provable outright
is an open question, posed by Boney and Brooke-Taylor. We will obtain the
following partial answer: 
\begin{repthm}{AlmostSCDichotomy}[SCH] 
If the least almost strongly compact cardinal has uncountable cofinality, it is
strongly compact.
\end{repthm}
It is not true in general that every almost strongly compact cardinal is
strongly compact, since any limit of strongly compact cardinals is almost
strongly compact, while every strongly compact cardinal is regular. 
UA does imply that every {\it successor} almost strongly compact cardinal is strongly
compact. Here we will show that this is almost a theorem of ZFC: 
\begin{repthm}{SuccessorSC}
For any ordinal \(\alpha\), if the \((\alpha+1)\)-st almost strongly compact limit
cardinal has uncountable cofinality, it is strongly compact. 
\end{repthm}
We must say ``limit cardinal'' because the successor of any strongly compact cardinal
is almost strongly compact.
\subsection{Cardinal preserving embeddings}
Next, in \cref{CardinalPreservingSection}, 
we take up the question of cardinal preserving embeddings, 
posed by Caicedo \cite{Caicedo}. If \(M\) is an inner model, an
elementary embedding \(j : V\to M\) is said to be {\it cardinal preserving} (up
to \(\lambda\)) if every cardinal of \(M\) (less than \(\lambda\)) is a cardinal
in \(V\). 

Caicedo asked whether cardinal preserving embeddings exist. The
Ultrapower Axiom implies that they do not. In fact, under UA, if \(\lambda\) is
an aleph fixed point and \(j : V\to M\) is an elementary embedding that fixes
\(\lambda\) and is cardinal preserving up to \(\lambda\), then
\(V_\lambda\subseteq M\). Since every elementary embedding
has an \(\omega\)-closed unbounded class of fixed points, 
it follows that under UA, no elementary embedding
\(j : V\to M\) can be fully cardinal preserving: otherwise \(V_\lambda\subseteq M\)
for a proper class of \(\lambda\), violating the Kunen inconsistency \cite{Kunen},
which states that there is no 
elementary embedding from \(V\) to \(V\).

We show that one can refute the existence of cardinal preserving embeddings 
from large cardinal axioms alone:
\begin{repthm}{NoStronglyDiscont}
Suppose there is a proper class of strongly compact cardinals. Then there are no
cardinal preserving embeddings.
\end{repthm}
This theorem can be viewed as a version of the Kunen inconsistency,
but the proof is completely different from all of the usual proofs of Kunen's
theorem.

\subsection{Definability from ultrafilters}
Finally, \cref{HODSection} studies the structure of ordinal 
definability under large cardinal
assumptions. The most prominent
question here is Woodin's HOD Conjecture \cite{Woodin}. It turns out that UA implies the HOD
Conjecture,\footnote{Technically UA implies the HOD Hypothesis. The HOD Conjecture
is that the HOD Hypothesis is provable in ZFC.} and more:
\begin{thm*}[UA]
If there is a supercompact cardinal, then \(V\) is a generic extension of
\(\textnormal{HOD}\).\qed 
\end{thm*}
The proof appears in the author's thesis \cite[Theorem 6.2.8]{UA}.

It is impossible to prove that \(V\) is a generic extension of
\(\textnormal{HOD}\) from any of the standard large cardinal hypotheses. We will
instead consider a generalization of HOD.
\begin{defn*}
    Let \(\kappa\text{-OD}\) denote the class of sets definable from a
    \(\kappa\)-complete ultrafilter over an ordinal, and let \(\kappa\text{-HOD}\) denote the
    class of hereditarily \(\kappa\)-OD sets.\footnote{
        Note that \(x\) is \(\kappa\)-OD if and only if \(x\) is in
        \(\textnormal{OD}_U\) for some \(\kappa\)-complete ultrafilter \(U\) over an
        ordinal, so \(\kappa\text{-OD}\) and \(\kappa\text{-HOD}\) are first-order
        definable.}
\end{defn*}

An ultrafilter over a set \(X\) can be thought of as a generalized element of
\(X\). From this perspective,
an ultrafilter over an ordinal is a generalized ordinal.
For this reason, definability from an ultrafilter over an ordinal seems to be 
a natural extension of ordinal definability.

Arguably, the more complete an ultrafilter over an ordinal is, 
the more it should resemble an ordinal. Thus as \(\kappa\) increases,
\(\kappa\)-HOD should become more like \(\textnormal{HOD}\);
for example, the \(\Omega\)-complete ultrafilters (i.e., those
that are \(\kappa\)-complete for all cardinals \(\kappa\)) are just 
the principal ultrafilters over ordinals, 
which are essentially just ordinals.
Therefore \(\Omega\text{-HOD}\) is just the usual HOD.
On the other hand, \(\omega\)-HOD, the class of sets definable from an arbitrary
ultrafilter over an ordinal, turns out to be equal to \(V\) (\cref{OmegaThm}).
The remaining models 
\(\langle \kappa\text{-HOD}\rangle_{\kappa\in \text{Card}}\)
form a decreasing sequence of structures between \(V\) 
and \(\textnormal{HOD}\).

Standard arguments show that for any cardinal \(\kappa\), \(\kappa\)-HOD is an inner model
of ZF. A little bit more surprisingly, if \(\kappa\) is strongly compact, then \(\kappa\)-HOD satisfies the
Axiom of Choice. 

It is consistent all known large cardinal axioms that \(V \neq \kappa\text{-HOD}\)
for any uncountable cardinal \(\kappa\), since this holds after adding a Cohen real
(\cref{Cohen}). We will show, however, that \(V\) is almost equal to \(\kappa\)-HOD. 
If \(M\) is an inner model of ZFC, 
\(M\) is said to be a {\it ground} of \(V\) if there is a partial order
\(\mathbb P\in M\) and an \(M\)-generic filter \(G\) on \(\mathbb P\)
such that \(M[G] = V\).
\begin{repthm}{Ground}
Suppose \(\kappa\) is strongly compact. Then
\(\kappa\textnormal{-HOD}\) is a ground of \(V\).
\end{repthm}
It follows that for all sufficiently large cardinals \(\lambda\),
\((\lambda^+)^{\kappa\textnormal{-HOD}} = \lambda^+\), \((2^\lambda)^{\kappa\textnormal{-HOD}} = 2^\lambda\),
\(\kappa\textnormal{-HOD}\) correctly computes stationary subsets of \(\lambda\),
large cardinals are transferred into and out of \(\kappa\)-HOD, etc. 
In fact, this is true for all \(\lambda \geq (2^\kappa)^+\).
Therefore unlike HOD, \(\kappa\)-HOD is provably very similar to \(V\).
The model \(\kappa\textnormal{-HOD}\) is, as far as we know, the first
nontrivial example of a ground of \(V\) that is not defined in terms of set
theoretic geology.

Last of all, we prove the following theorem:
\begin{repthm}{CheapHOD}
Suppose \(\kappa\) is supercompact. Then \(\kappa\) is supercompact in
\(\kappa\textnormal{-HOD}\).
\end{repthm}
Since supercompactness is defined in terms of \(\kappa\)-complete normal fine
ultrafilters, which are necessarily \(\kappa\)-OD,
\cref{CheapHOD} may not seem very surprising.
The issue one must overcome, however, is that these ultrafilters
might not concentrate on \(\kappa\)-HOD
and therefore might not witness that \(\kappa\) is 
supercompact in \(\kappa\)-HOD.
This corresponding question for
strongly compact cardinals remains open.
\section{Preliminaries}
We put down some definitions and notational conventions, most of which are
completely standard.
\subsection{Ultrafilters}
\begin{defn}
    If \((\mathbb P,\leq)\) is a partial order, a proper subset \(F\subseteq \mathbb P\) 
    is a {\it filter} on \(\mathbb P\) if it is closed upwards under \(\leq\)
    and for any \(p,q\in F\), there is some \(r\in F\) with \(r\leq p\) and \(r\leq q\). 
    A filter \(U\) on \(\mathbb P\) is an {\it ultrafilter} on
    \(\mathbb P\) if it is \(\subseteq\)-maximal among all filters on \(\mathbb P\).
\end{defn}
We are really only interested in the following special case:
\begin{defn}
    Suppose \(M\) is a model of set theory and \(X\in M\). We say \(U\subseteq
    P^M(X)\) is an {\it \(M\)-filter} (resp.\ {\it \(M\)-ultrafilter}) over \(X\) if
    \(U\) is a filter (resp.\ ultrafilter) on the partial order \((P^M(X),\subseteq^M)\). 
\end{defn}
A fundamental concept in the theory of large cardinals is the completeness of an ultrafilter.
We will need the generalization of this concept to \(M\)-ultrafilters.
\begin{defn}
    Suppose \(M\) is a model of set theory, \(U\) is an \(M\)-ultrafilter,
    \(\rho\) is an \(M\)-cardinal, and \(\kappa\) is a cardinal. 
    \begin{itemize}
        \item \(U\) is {\it \(M\)-\(\rho\)-complete}
        if for any \(\sigma\subseteq U\) with \(\sigma\in M\) and \(|\sigma|^M <
        \rho\), \(\bigcap\sigma\in U\).
        \item  \(U\) is \(M\)-\(\kappa\)-complete if 
        for any \(\sigma\subseteq U\) with \(\sigma\in M\) and \(|\sigma| < \kappa\),
        \(\bigcap \sigma\in U\).\footnote{The notation \(|\sigma| < \kappa\) 
        cannot be taken literally when \(M\) is illfounded. 
        We really mean that the {\it extension} of \(\sigma\), i.e., the set
        \(\{x\in M : M\vDash x\in \sigma\}\), has cardinality less than \(\kappa\).
        Going forward, we will identify elements of illfounded models with their extensions
        without comment.}
    \end{itemize}
\end{defn}
If \(U\) is a \(V\)-ultrafilter over \(X\), we say that \(U\) is an {\it ultrafilter
over \(X\),} and if \(U\) is \(V\)-\(\kappa\)-complete, we say \(U\) is
{\it \(\kappa\)-complete.}

We denote the ultrapower of a model \(P\) by an \(P\)-ultrafilter \(U\) by
\[j_U : P\to M_U^P\] The ultrapower of \(V\) by an ultrafilter \(U\) is denoted
\(j_U : V\to M_U\).

The following terminology is probably self-explanatory:
\begin{defn}
    Suppose \(M\) is a model of set theory and \(W\) is an \(M\)-ultrafilter.
    Then \(W\) is an {\it ultrafilter of \(M\)} if \(W\in M\).
\end{defn}

We now turn to some basic combinatorial definitions.
\begin{defn}
    Suppose \(M\) is a model of set theory and \(X\in M\). An \(M\)-ultrafilter
    \(U\) over \(X\) is {\it uniform} if every set in \(U\) has \(M\)-cardinality \(|X|^M\).
\end{defn}
Note that if \(M\) is a wellfounded model
of ZFC, then for any \(M\)-ultrafilter \(U\), there is some \(A\in U\) such that
\(U\cap P^M(A)\) is uniform. In particular, this holds for any \(V\)-ultrafilter,
so in the theory of ultrafilters, one can usually work 
with uniform ultrafilters with no loss of generality.

A notion similar to uniformity, but distinct from it, is fineness:
\begin{defn}
    An ultrafilter \(U\) over a family of sets \(F\) is {\it fine}
    if for all \(x\in \bigcup F\), the set \(\{A\in F : x\in A\}\)
    belongs to \(U\).
\end{defn}
This is a slight generalization of the standard definition of fineness.
Note that an ultrafilter \(U\) over an ordinal \(\alpha\)
is fine if and only if every set in \(U\) is cofinal in \(\alpha\).

\begin{defn}\label{PushDef}
    Suppose \(f\) is a function, \(U\) is an ultrafilter over a set \(X\), and
    \(Y\) is a set such that \(f^{-1}[Y]\cap X\in U\). The {\it pushforward of
    \(U\) under \(f\) over \(Y\)} is the ultrafilter defined by \(f_*(U) =
    \{A\subseteq Y : f^{-1}[A]\cap X\in U\}\).
\end{defn}
Our notation for pushforwards ignores the choice of \(Y\), which we ask the
reader to infer from context. For notational convenience, we allow that
\(\text{dom}(f)\neq X\) and \(\text{ran}(f)\neq Y\), and instead require just
that \(f\) is defined \(U\)-almost everywhere and sends \(U\)-almost every
element of \(X\) to an element of \(Y\). This is not really an important point.

What {\it is} important is the relationship between pushforwards and
derived ultrafilters.
\begin{defn}
If \(j : M\to N\) is an elementary embedding, \(X\in M\), and \(a\in j(X)\), then the
{\it \(M\)-ultrafilter over \(X\) derived from \(j\) using \(a\)} is the set
\(\{A\in P^M(X) : a\in j(A)\}\).
\end{defn}

\begin{prp}\label{PushforwardLma}
Suppose \(U\) and \(W\) are ultrafilters over sets \(X\) and \(Y\)
and \(f\) is a function such that \(f^{-1}[Y]\cap X\in U\).
Then the following are equivalent:
\begin{enumerate}[(1)]
    \item \(f_*(U) = W\).
    \item \(W\) is the ultrafilter on \(Y\) derived from \(j_U\) using \([\textnormal{f}]_U\).
    \item  There exists an elementary embedding \(k : M_W\to M_U\) such that \(k\circ j_W = j_U\) 
        and \(k([\textnormal{id}]_W) = [f]_U\).\qed
\end{enumerate}
\end{prp}
Notice that there is at most one embedding witnessing (3).
\subsection{The approximation and cover properties}
For our results, it is important to define covering properties for models that
are not necessarily wellfounded.
\begin{defn}\label{CovDef} Suppose \(M\) is a model of set theory, \(X\in
M\) is a set, \(\rho\) is an \(M\)-cardinal, and \(\kappa\) is a cardinal.
\begin{itemize}
    \item \(M\) has the {\it \((\kappa,\rho)\)-cover property} if for all
    \(\sigma\subseteq M\) with \(|\sigma| < \kappa\), there is some \(\tau\in M\)
    with \(|\tau|^M < \rho\) such that \(\sigma\subseteq \tau\).
    \item {\it \(\kappa\)-cover property} if for all \(\sigma\subseteq M\) with \(|\sigma|
    < \kappa\), there is some \(\tau\in M\) with \(|\tau| < \kappa\) such that
    \(\sigma\subseteq \tau\).
\end{itemize}
\end{defn}
We will also discuss the Hamkins approximation property \cite{HamkinsApprox},
but we pass over the illfounded case:
\begin{defn}\label{AppxPrpDef} 
    Suppose \(M\) is a model of set theory and \(\kappa\) is a cardinal.
    \begin{itemize}
        \item A set \(A\subseteq M\) is {\it \(\kappa\)-approximated by \(M\)} if
              for all \(\sigma\in M\) with \(|\sigma| < \kappa\), \(A\cap \sigma\in
              M\).
        \item \(M\) has the {\it \(\kappa\)-approximation property} if every set
              that is \(\kappa\)-approx\-imated by \(M\) belongs to \(M\).
    \end{itemize}
\end{defn}
These two properties combined define the notion of a pseudoground:
\begin{defn}\label{Pseudoground}
Suppose \(M\subseteq N\) are transitive models of ZFC and \(\kappa\) is an \(N\)-cardinal. 
We say \(M\) is a {\it \(\kappa\)-pseudoground} of \(N\)
if \(N\) satisfies that \(M\) has the \(\kappa\)-approximation and cover properties.\footnote{Formally this is
expressed in the structure \((N,M)\).}
We say \(M\) is a {\it pseudoground} of \(N\) if there is some \(N\)-cardinal \(\kappa\) such that
\(M\) is a \(\kappa\)-pseudoground of \(N\).
\end{defn}
We will refer to pseudogrounds of \(V\) simply as pseudogrounds.

Note that if \(M\) is a pseudoground of \(N\) then \(\text{Ord}\cap N\subseteq M\),
or in other words, \(M\) is an inner model of \(N\). In particular, \(M\) is not
an element of \(N\),
but it turns out that \(M\) must be definable over \(N\): 
\begin{thm}[Laver-Hamkins]\label{HamkinsUniqueness}
    Suppose \(M\) is a \(\kappa\)-pseudoground of \(N\).
    Then \(M\) is the unique \(\kappa\)-pseudoground \(P\) of \(N\)
    such that \(P\cap H(\kappa^{+N}) = M\cap H(\kappa^{+N})\) and
    \(M\) is \(\Delta_2\)-definable over \(N\) from the
    parameter \(M\cap H(\kappa^+)\).\qed
\end{thm}

The following is Woodin's Universality Theorem for pseudogrounds:
\begin{thm}[Woodin]
    Suppose \(M\) is a \(\kappa\)-pseudoground and \(E\) is an \(M\)-extender
    of length \(\nu\) whose critical point is at least \(\kappa\).
    If \(j_E(A)\cap [\nu]^{<\omega}\in M\) for all \(A\in M\), 
    then \(E\in M\).\qed
\end{thm}
The Hamkins Universality Theorem shows that for nice embeddings, one does not
even have to assume closure under the extender:
\begin{thm}[Hamkins]\label{HamkinsUniversality}
    Suppose \(M\) is a \(\kappa\)-pseudoground. 
    \begin{itemize}
        \item Every \(\kappa\)-complete \(M\)-ultrafilter belongs to \(M\).
        \item If \(E\) is an extender with critical point greater than 
            \(\kappa\) such that \(M_E\) is closed under
            \(\kappa\)-sequences, then \(E\cap M\in M\). \qed
    \end{itemize}
\end{thm}

\begin{thm}[Hamkins-Reitz]
Suppose \(\kappa\) is a cardinal and \(M\) is a \(\kappa\)-pseudoground. Then \(M\) is
a \(\lambda\)-pseudoground for all \(\lambda\geq \kappa\).\qed
\end{thm}
\subsection{Compactness principles}\label{CompactnessSection}
In this section, we define various notions of strong compactness, the
most famous of which is of course due to Tarski \cite{Tarski}, and the rest of which
were introduced by Bagaria-Magidor \cite{MagidorBagaria}.
\begin{defn}
Suppose \(\delta\leq \kappa \leq\lambda\) are cardinals. Then \(\kappa\) is
\emph{\((\delta,\lambda)\)-strongly compact} if there is a \(\delta\)-complete
fine ultrafilter over \(P_\kappa(\lambda)\).
\end{defn}
This principle is degenerate in the sense that if \(\kappa\) is
\((\delta,\lambda)\)-strongly compact, then all ordinals above \(\kappa\) are
\((\delta,\lambda)\)-strongly compact.
\begin{defn}
Suppose \(\delta\leq \kappa\leq \lambda\) are cardinals. 
\begin{itemize}
    \item \(\kappa\) is \emph{\((\delta,\infty)\)-strongly compact} if it is
    \((\delta,\gamma)\)-strongly compact for all cardinals \(\gamma\geq
    \kappa\).
    \item  \(\kappa\) is \emph{\(\lambda\)-strongly compact} if it is
    \((\kappa,\lambda)\)-strongly compact.
    \item \(\kappa\) is \emph{strongly compact} if it is
    \((\kappa,\infty)\)-strongly compact.
    \item \(\kappa\) is \emph{almost \(\lambda\)-strongly compact} if it is
    \((\gamma,\lambda)\)-strongly compact for all cardinals \(\gamma < \kappa\).
    \item \(\kappa\) is \emph{almost strongly compact} if it is almost
    \(\eta\)-strongly compact for all cardinals \(\eta\geq \kappa\).
\end{itemize}
\end{defn}
These principles can be reformulated in terms of either
the filter extension property, elementary embeddings, or uniform ultrafilters on cardinals.
We will actually use all four characterizations below without much comment.
\begin{thm}[Solovay, Ketonen]\label{KetonenThm} Suppose \(\delta\leq \kappa\leq
    \lambda\) are cardinals. Then the following are equivalent:
    \begin{itemize}
        \item \(\kappa\) is \((\delta,\lambda)\)-strongly compact.
        \item There is an elementary embedding \(j : V\to M\) with
              \(\textnormal{crit}(j) \geq \delta\) such that \(M\) has the
              \((\lambda^+,j(\kappa))\)-cover property.
        \item Every \(\kappa\)-complete filter that is generated by at most
              \(\lambda\) sets extends to a \(\delta\)-complete ultrafilter.
        \item Every regular cardinal in the interval \([\kappa,\lambda]\)
              carries a \(\delta\)-complete uniform ultrafilter.\qed
    \end{itemize}
\end{thm}
We also use the following theorem, which is essentially due to Solovay:
\begin{thm}[Solovay]\label{AlmostSolovay}
    The Singular Cardinals Hypothesis
    holds above the least almost strongly compact cardinal \(\kappa\): 
    if \(\lambda \geq \kappa\) is a singular cardinal, then 
    \[\lambda^{\textnormal{cf}(\lambda)} = \textnormal{max}(2^{\textnormal{cf}(\lambda)},\lambda^+)\pushQED \qed\qedhere\popQED\]
\end{thm}
\section{Ultrafilters in ultrapowers}\label{UltrafilterSection}
Suppose \(D\) is an ultrafilter and \(W\) is an \(M_D\)-ultrafilter. It is often
useful to know whether \(W\) belongs to \(M_D\). The Ultrapower Axiom yields
many instances in which \(W\in M_D\) must occur for \(D\) and \(W\) countably
complete ultrafilters; this fact is leveraged to prove most of the
consequences of UA in \cite{UA}. But it turns out that in certain situations, one can prove
that \(W\in M_D\) from large cardinal axioms alone.
 
The idea is that if one can \(W\) extend to a sufficiently complete \(V\)-ultrafilter \(W^*\), 
then using a result known as Kunen's {\it commuting ultrapowers lemma}
one obtains that \(j_{W^*}\restriction M_D\) is definable over \(M_D\),
and hence \(W\) belongs to \(M_D\).
In \cref{CommutingUltrapowersSection}, we give a proof of Kunen's result.
(The reason we include this is to verify that the proof goes through
in the case that \(D\) is countably incomplete, which we need in order to answer
to Silver's question above a strongly compact cardinal.)

In \cref{SuffCompleteSection}, we prove that if \(\kappa\) is a strong limit cardinal
and \(W\) is \(\kappa\)-complete 
with respect to sets in \(M_D\), 
then \(W\) generates a \(\kappa\)-complete filter in \(V\). Thus if
\(\kappa\) is strongly compact, \(W\) extends to a \(\kappa\)-complete
ultrafilter \(W^*\), and so by the observation in the previous paragraph,
one can conclude that \(W\in M_D\).

Finally, \cref{AppxSection} is devoted to applications of the results of \cref{SuffCompleteSection}
to the theory of {\it pseudogrounds,} a generalization due to Hamkins
\cite{Hamkins} of the concept of a set forcing ground of \(V\) 
that appears to have a deep relationship with the theory of inner models for supercompact cardinals. 
These applications digress from the main thread of this paper,
and are not strictly speaking necessary to prove our main results.
What we show is that if \(\kappa\) is strongly compact, then 
\(\kappa\)-pseudogrounds are characterized by their most basic properties:
\begin{repthm}{CKappa}
        Suppose \(\kappa\) is strongly compact and \(M\) is an inner model. Then the
        following are equivalent:
        \begin{enumerate}[(1)]
        \item \(M\) is a \(\kappa\)-pseudoground.
        \item \(\kappa\) is strongly compact in \(M\) and the following hold:
            \begin{itemize}
                \item Every \(\kappa\)-complete ultrafilter
                over a set in \(M\) extends an ultrafilter of \(M\).
                \item Every \(\kappa\)-complete ultrafilter of \(M\)
                extends to a \(\kappa\)-complete ultrafilter of \(V\).
            \end{itemize}
        \item Every regular cardinal of \(M\) above \(\kappa\) has cofinality
            at least \(\kappa\) and every \(\kappa\)-complete ultrafilter
            over a set in \(M\) extends an ultrafilter of \(M\).
        \end{enumerate}
\end{repthm}

\subsection{Commuting ultrafilters and ultrapowers}\label{CommutingUltrapowersSection}
The following definition explains how elementary embeddings act on amenable classes.
\begin{defn}
Suppose \(M\) is a model of set theory. A class \(A\)
is {\it amenable to \(M\)} if \(A\subseteq M\) and 
\(A\cap x\in M\) for all \(x\in M\).

An elementary embedding \(j : M\to N\) is {\it cofinal}
if for every \(a\in N\), there is some \(X\in M\) such that \(a\in j(X)\). 

If \(j : M\to N\) is a cofinal elementary embedding, and \(A\) is 
an amenable class of \(M\),
then \(j(A) = \bigcup_{x\in P} j(A\cap x).\)
\end{defn}
If \(j' : M' \to N'\) extends \(j: M \to N\), and \(j'\restriction M : M \to j'(M)\) is 
a cofinal embedding, then \(j'(A) = j(A)\) for all \(A\) in \(M'\) amenable to \(M\).

If \(j_0 : M\to N_0\) is a cofinal elementary embedding 
and \(j_1 : M\to N_1\) is an amenable elementary embedding, then 
\(j_0(j_1): N_0\to j_0(N_1)\) is an elementary embedding.
\begin{defn}
    Suppose \(j_0 : V\to M_0\) and \(j_1 : V\to M_1\) are cofinal elementary
    embeddings. We say \(j_0\) and \(j_1\) {\it commute} if there is an
    isomorphism \(k : j_0(M_1)\to j_1(M_0)\) such that 
    \[j_1\restriction M_0 = k \circ j_0(j_1)
    \text{ and }j_0\restriction M_1 = k^{-1}\circ j_1(j_0)\]
\end{defn}
Note that we do not assume that
\(M_0\) and \(M_1\) are wellfounded.
If \(M_0\) and \(M_1\) are {\it transitive,} then \(k\) must be the identity, and
hence \(j_0\) and \(j_1\) commute if and only if \[j_0(j_1) = j_1\restriction
M_0\text{ and }j_1(j_0) = j_0\restriction M_1\]
\begin{comment}
The following lemma is a version of Kunen's Commuting Ultrapowers Lemma:
\begin{lma}
    Suppose \(j_0 : V\to M_0\) and \(j_1 : V\to M_1\) are cofinal elementary
    embeddings. Suppose \(Z\subseteq M_1\) is such that every element of \(M_1\)
    is definable in \(M_1\) from parameters in \(j_1[V]\cup Z\). Assume that
    there is an isomorphism \(k: j_0(M_1)\to j_1(M_0)\) such that \(k\circ
    j_0(j_1) = j_1\restriction M_0\) and \(k\circ j_0\restriction Z =
    j_1(j_0)\restriction Z\). Then \(j_1(j_0) = k\circ j_0\restriction M_1\).
    \begin{proof}
        Note that \[k\circ (j_0\restriction M_1)\circ j_1 = k\circ j_0(j_1)\circ
        j_0 = (j_1\restriction M_0)\circ j_0 = j_1(j_0)\circ j_1\] In other
        words, \(k\circ j_0\restriction j_1[V] = j_1(j_0)\restriction j_1[V]\).
        On the other hand, by assumption \(k\circ j_0\restriction Z=
        j_1(j_0)\restriction j_1(V_\xi)\), and so \(k\circ j_0\) and
        \(j_1(j_0)\) agree on \(j_1[V]\cup Z\). Since \(k\circ j_0\) and
        \(j_1(j_0)\) are elementary and every element of \(M_1\) is definable in
        \(M_1\) from parameters in \(j_1[V]\cup Z\), we must have \(k\circ j_0 =
        j_1(j_0)\).
    \end{proof}
\end{lma}
\end{comment}

For ultrafilters \(U\) and \(W\), whether \(j_U\) and \(j_W\) commute
is influenced by the relationship between
the filter product \(U\times W\) and the ultrafilter product \(U\otimes W\)
of \(U\) and \(W\).
\begin{defn}
    Suppose \(U\) and \(W\) are ultrafilters over sets \(X\) and \(Y\). 
    \begin{itemize}
        \item  {\it \(W\) is \(U\)-complete} if for any sequence \(\langle B_x :
        x\in X\rangle\subseteq W\), there is some \(A\in U\) such that
        \(\bigcap_{x\in A}B_x\in W\).
        \item The {\it filter product of \(U\) and \(W\)} is the filter
        \(U\times W\) generated by sets of the form \(A\times B\) for \(A\in U\)
        and \(B\in W\).
        \item The {\it ultrafilter product of \(U\) and \(W\)} is the
        ultrafilter
        \[U\otimes W = \{A\subseteq X\times Y : \forall^Ux\ \forall^Wy\ (x,y)\in
        A\}\]
    \end{itemize}
\end{defn}
The filter product is commutative up to canonical isomorphism, but in general
the ultrafilter product is not.

\begin{thm}[Blass]\label{BlassThm}
Suppose \(U\) and \(W\) are ultrafilters. Then the following are equivalent:
\begin{enumerate}[(1)]
\item \(W\) is \(U\)-complete.
\item \(j_U[W]\) generates \(j_U(W)\).
\item \(U\times W\) is an ultrafilter.
\item \(U\times W = U\otimes W\).
\end{enumerate}
\begin{proof}
Let \(X\) and \(Y\) be the underlying sets of \(U\) and \(W\).

{\it (1) implies (2):} Fix 
\([B_x]_U \in U\). The \(U\)-completeness of \(W\)
yields that for some \(A\in U\), \(\bigcap_{x\in A} B_x\in W\).
Note that \(j_U\left(\bigcap_{x\in A} B_x\right)\subseteq [B_x]_U\) since 
\(\bigcap_{x\in A} B_x\subseteq B_x\)
for \(U\)-almost all \(x\in X\). Thus \([B_x]_U\) contains
and element of \(j_U[W]\), as desired.

{\it (2) implies (3):} Fix \(R\subseteq X\times Y\).
We will prove that either \(R\) or its complement contains a set in \(U\).
Let \(R_x = \{y\in Y : (x,y)\in R\}\).
Assume without loss of generality that \(R_x\in W\) for \(U\)-almost all \(x\in X\).
Then \([R_x]_U\in j_U(W)\), so by (2), there is some \(B\in W\) such that \(j_U(B)\subseteq [R_x]_U\).
Fix \(A\in U\) such that for all \(x\in A\),
\(B\) is contained in \(R_x\).
Then \(A\times B\subseteq R\) and \(A\times B\in U\times W\).

{\it (3) implies (4):} This is trivial since by definition 
\(U\times W \subseteq U\otimes W\), so if \(U\times W\) is an ultrafilter,
then \(U\times W = U\otimes W\) by maximality.

{\it (4) implies (1):} Fix \(\langle B_x\rangle_{x\in X}\subseteq W\).
Let \(R = \{(x,y)\in X\times Y : y\in B_x\}\). Then \(R\in U\otimes W\)
by definition. Therefore \(R\in U\times W\) by
(4), so fix \(A\in U\) and \(B\in W\) such that
\(A\times B\subseteq R\). Then for all \(x\in A\),
\(B\subseteq B_x\). In other words, \(B\subseteq \bigcap_{x\in A} B_x\),
so since \(B\in W\), \(\bigcap_{x\in A} B_x\in W\). This shows that \(W\)
is \(U\)-complete.
\end{proof}
\end{thm}
The equivalence of (1) and (3) in \cref{BlassThm} implies that an ultrafilter
 \(W\) is \(U\)-complete if and only if
\(U\) is \(W\)-complete, which is a bit surprising given the original definition.

\begin{lma}\label{CommuteEquiv}
Suppose \(U\) and \(W\) are ultrafilters over \(X\) and \(Y\).
The following are equivalent:
\begin{enumerate}[(1)]
    \item \(j_U\) and \(j_W\) commute.
    \item Let \(\textnormal{flip}(x,y) = (y,x)\).
    Then \(\textnormal{flip}_*(U\otimes W) = W\otimes U\).\footnote{
        Recall that \(f_*(U) = \{A : f^{-1}[A]\in U\}\);
        see \cref{PushDef}.
    }
    \item The quantifiers associated to \(U\) and \(W\) commute. That is,
        for any predicate \(R\) on \(X\times Y\),
        \[\forall^U x\ \forall^W y\ R(x,y)\iff \forall^W y\ \forall^U x\ R(x,y)\]
\end{enumerate}
\begin{proof}
{\it (1) if and only if (2):} 
There is a natural isomorphism between \(M_{U\otimes W}\) and
\(j_U(j_W)(M_U)\) sending a point \([f]_{U\otimes W}\), where
\(f : X\times Y \to V\), to the point \([[\lambda f]_U]_{j_U(W)}\)
where \(\lambda f : X\to V^Y\) is the function defined by 
\(\lambda f(x)(y) = f(x,y)\).
For notational convenience,
we will identify the two models via this isomorphism.

This identification results in the following equalities: 
\begin{align*}
    j_{U\otimes W} &= j_U(j_W)\circ j_U\\
    [\text{id}]_{U\otimes W} &= (j_U(j_W)([\text{id}]_U), j_U([\text{id}]_W))
\end{align*}
Under the corresponding identification of \(M_{W\otimes U}\) with
\(j_W(j_U)(M_W)\),
\begin{align*}
    j_{W\otimes U} &= j_W(j_U)\circ j_W\\
    [\text{id}]_{W\otimes U} &= (j_W(j_U)([\text{id}]_W), j_W([\text{id}]_U))
\end{align*}

Given these equalities and \cref{PushforwardLma},
the function \(\text{flip}(x,y) = (y,x)\) satisfies \(\text{flip}_*(U\otimes W) = W\otimes U\)
if and only if there is an elementary embedding \(k : M_{W\otimes U}\to M_{U\otimes W}\)
satisfying
\begin{align}k \circ j_W(j_U)\circ j_W &= j_U(j_W)\circ j_U\label{CommuteEq}\\
k(j_W(j_U)([\text{id}]_W), j_W([\text{id}]_U) &= \text{flip}(j_U(j_W)([\text{id}]_U), j_U([\text{id}]_W))\label{SeedEq} 
\end{align}

We claim that an embedding satisfies \cref{CommuteEq} and \cref{SeedEq} if and only if 
it is an isomorphism such that
\(k\circ j_W(j_U) = j_U\restriction M_W\),
and \(k^{-1}\circ j_U(j_W) = j_W\restriction M_U\), or in other words, \(j_U\) and \(j_W\) commute.

For the forwards direction, assume \(k\) satisfies \cref{CommuteEq} and \cref{SeedEq}.
We claim \(k\) is surjective. Let
\[S = j_U(j_W)\circ j_U[V] \cup \{j_U(j_W)([\text{id}]_U), j_U([\text{id}]_W)\}\] 
Then \L o\'s's Theorem implies every element of
\(j_U(j_W)(M_U)\) is definable in \(j_U(j_W)(M_U)\) 
from parameters in \(S\).
But \cref{CommuteEq} and \cref{SeedEq} imply that
\(S\subseteq \text{ran}(k)\). Since \(\text{ran}(k)\) is closed under 
definability in \(j_U(j_W)(M_U)\), every point in \(j_U(j_W)(M_U)\) 
is in \(\text{ran}(k)\), so \(k\) is surjective.
It follows that \(k\) is an isomorphism.

To see that \(k\circ j_W(j_U) = j_U\restriction M_W\), notice that
\(k\circ j_W(j_U)\) agrees on \(j_W[V]\) with \(j_U\) by 
\cref{CommuteEq}, 
and \(k(j_W(j_U)([\text{id}]_W)) = j_U([\text{id}]_W)\).
Hence \(k\circ j_W(j_U)\) and \(j_U\) agree on 
\(j_W[V]\cup \{[\text{id}]_W\}\).
Since every point in \(M_W\) is definable in \(M_W\) from 
parameters in \(j_W[V]\cup \{[\text{id}]_W\}\), 
\(k\circ j_W(j_U) = j_U\restriction M_W\) by elementarity.
The fact that \(k^{-1}\circ j_U(j_W) = j_W\restriction M_U\) is proved by a similar argument.

The reverse direction of the claim is very similar, so we omit the proof.
We also omit the proof of the equivalence of (2) and (3), since there are
no ideas there, and anyway (3) was included only for aesthetic reasons.
\end{proof}
\end{lma}

We now prove the Commuting Ultrapowers Lemma using Blass's result.
\begin{thm}[Kunen]\label{CompleteCommute} Suppose \(U\)
and \(W\) are ultrafilters such that \(W\) is \(U\)-complete. Then \(j_U\) and \(j_W\)
commute.
\begin{proof}
By \cref{BlassThm}, \(U\otimes W = U\times W\).
Therefore \(\text{flip}_*(U\otimes W) = \text{flip}_*(U\times W) = W\times U\).
Since \(W\times U\) is an ultrafilter, by \cref{BlassThm}
(with the roles of \(U\) and \(W\) exchanged), \(W\times U = W\otimes U\).
Putting these equations together, \(\text{flip}_*(U\otimes W) = W\otimes U\).
Applying \cref{CommuteEquiv}, \(j_U\) and \(j_W\) commute.
\end{proof}
\end{thm}
Whether the converse of \cref{CompleteCommute} is provable in ZFC an open question. The converse restricted 
to countably complete ultrafilters is an easy consequence 
of the Ultrapower Axiom. The author has also proved that the converse follows
from the Generalized Continuum Hypothesis. Thus another consequence of UA can be
verified by a classical axiom.

\subsection{Sufficiently complete \(M_D\)-ultrafilters are in \(M_D\)}\label{SuffCompleteSection}
In this subsection we prove our main theorem on the amenability of ultrafilters:
\begin{thm}\label{AmenableThm} Suppose \(\delta\) is a cardinal and
    \(\kappa\geq \delta\) is a strong limit cardinal.
    Suppose \(D\) is an ultrafilter over a set of size less than \(\delta\)
    and \(X\) is a set in \(M_D\). Suppose
    \(W\) is an \(M_D\)-\(\kappa\)-complete \(M_D\)-ultrafilter over \(X\in M_D\). Assume
    that every \(\kappa\)-complete filter over \(X\) extends to a \(\delta\)-complete
    ultrafilter. Then \(W\in M_D\).
\end{thm}

The main point is that in the situation of \cref{AmenableThm}, the
\(M_D\)-ultrafilter \(W\) can be extended to a \(\delta\)-complete ultrafilter:
\begin{prp}\label{ContinuumComplete} Suppose
    \(\kappa\) is a strong limit cardinal and \(M\) is a model of set theory with
    the \(\kappa\)-cover property.
    Suppose \(W\) is an \(M\)-\(\kappa\)-complete
    \(M\)-ultrafilter. Then \(W\) generates a \(\kappa\)-complete filter.
\end{prp}
This in turn follows from Kunen's analysis of weakly amenable ultrafilters,
which we state in a very general form:
\begin{thm}[Kunen]\label{KunenAmenable} Suppose \(M\) is a model of set theory,
\(U\) is an \(M\)-ultrafilter over \(X\in M\), and \(\iota\) is an
\(M\)-cardinal. Let \(j : M\to N\) be the ultrapower of \(M\) by \(U\). Then the
following are equivalent:
\begin{enumerate}[(1)]
\item For all \(\sigma\subseteq P^M(X)\) with \(\sigma\in M\) and \(|\sigma|^M =
      \iota\), \(U\cap \sigma \in M\).
\item For all \(B\in P^N(j(\iota))\), \(j^{-1}[B]\in M\).
\end{enumerate}
\begin{proof}
    {\it (1) implies (2):} Fix \(B\in P^N(j(\iota))\). Let \(f: X\to
    P^M(\iota)\) be a function in \(M\) such that \(B = [f]_U^M\). For \(\xi <
    \iota\), let \(A_\xi = \{x\in X : \xi \in f(x)\}\). Note that the sequence
    \(\langle A_\xi\rangle_{\xi < \iota}\) belongs to \(M\). Let \(\sigma =
    \{A_\xi : \xi < \iota\}\). Now \(j^{-1}[B] = \{\xi < \iota : A_\xi\in U\} =
    \{\xi < \iota : A_\xi\in U\cap \sigma\}\). But by (1), \(U\cap\sigma\in M\).
    Hence \(j^{-1}[B]\in M\).

    {\it (2) implies (1):} Fix \(\sigma\subseteq P^M(X)\) and a surjection \(f :
    \iota\to \sigma\) that belongs to \(M\). Let \(a = [\text{id}]_U\). Let \(B
    = \{\xi < j(\iota) : a \in j(f)(\xi)\}\). Clearly \(B\in P^N(j(\iota))\).
    Note that \(j(\xi)\in B\) if and only if \(a\in j(f(\xi))\), which happens
    if and only if \(f(\xi)\in U\). In other words, \(U\cap \sigma = \{f(\xi) :
    \xi\in  j^{-1}[B]\}\). By (2), \(j^{-1}[B]\in M\), and so \(U\cap \sigma\in
    M\).
\end{proof}
\end{thm}

\begin{proof}[Proof of \cref{ContinuumComplete}]
To show that \(W\) generates a \(\kappa\)-complete filter, it suffices to show
that for all \(\sigma\subseteq W\) with \(|\sigma| < \kappa\), \(\bigcap
\sigma\) is nonempty. By the \(\kappa\)-cover property, there is some
\(\tau\in M\) containing \(\sigma\) of cardinality less than \(\kappa\). 
Let \(\iota = |\tau|^M\).

Let \(j : M\to N\) be the ultrapower of \(M\) by \(W\). Since \(\kappa\) is a
strong limit cardinal, \(W\) is \(M\)-\(((2^\iota)^+)^M\)-complete. 
This implies that \(j[P^M(\iota)] =
j(P^M(\iota)) = P^N(j(\iota))\). (Recall that an \(M\)-ultrafilter \(U\) is
\(M\)-\(\delta\)-complete if and only if \(j(X) = j[X]\) for every \(X\in M\)
with \(|X|^M < \delta\).) In particular, 
for any \(B\in P^N(j(\iota))\), there
is some \(\bar B\in M\) with \(j(\bar B) = B\), so \(j^{-1}[B]\in M\) 
since \(j^{-1}[B] = \bar B\).

Applying \cref{KunenAmenable}, it follows that \(W\cap \tau\in M\). Since \(W\)
is \(M\)-\(\rho\)-complete, \(\bigcap (W\cap \tau)\) is nonempty. But
\(\sigma\subseteq W\cap \tau\), so \(\bigcap \sigma\) is nonempty, as desired.
\end{proof}

To obtain the cover hypothesis in \cref{ContinuumComplete}, we establish a
general fact about the covering properties of ultrapowers.
\begin{prp}\label{UltraCov} Suppose \(\kappa\) is a 
    strong limit cardinal and \(D\) is an ultrafilter 
    over a set \(X\) of size less than \(\kappa\). 
    Then \(M_D\) has the \(\kappa\)-cover property.
    \begin{proof}
        Fix \(\sigma\subseteq M_D\) with \(|\sigma| < \kappa\).
        Let \(\delta = |\sigma|\) and
        choose functions \(\langle f_\alpha\rangle_{\alpha < \delta}\)
        such that 
        \[\sigma = \{[f_\alpha]_U: {\alpha < \delta}\} = \{j_U(f_\alpha)([\text{id}]_U) : {\alpha < \delta}\}\]
        Let
        \(\langle g_\beta\rangle_{\beta < j_D(\delta)} 
        = j_D(\langle f_\alpha\rangle_{\alpha < \delta})\).
        Then \(\{j_U(f_\alpha) : {\alpha < \delta}\}\subseteq \{g_\beta : {\beta < j_D(\delta)}\}\),
        so \(\sigma\subseteq \{g_\beta([\text{id}]_U) : \beta < j_D(\delta)\}\).
        Clearly \(\{g_\beta([\text{id}]_U) : \beta < j_D(\delta)\}\in M_D\) and
        has cardinality at most \(|j_D(\delta)| \leq \delta^{|X|} < \kappa\)
        since \(\kappa\) is a strong limit cardinal.
    \end{proof}
\end{prp}

\cref{AmenableThm} is now a matter of citing the preceding results in
the right order.
\begin{proof}[Proof of \cref{AmenableThm}]
    By \cref{UltraCov}, \(M_D\) has the
    \(\kappa\)-cover property, so by \cref{ContinuumComplete}, 
    the \(M_D\)-\(\kappa\)-complete
    \(M_D\)-ultrafilter \(W\) is \(\kappa\)-complete,
    or in other words, it generates a \(\kappa\)-complete filter.

    Let \(F\) be the \(\kappa\)-complete filter generated by \(W\). The filter
    \(F\) extends to a \(\delta\)-complete ultrafilter \(U\) by our large
    cardinal hypothesis. Now we apply the Commuting Ultrapowers Lemma
    (\cref{CompleteCommute}) to conclude that \(U\cap M_D\) belongs to \(M_D\).
    More precisely, the Commuting Ultrapowers Lemma yields an isomorphism \(k :
    j_U(M_D)\to j_D(M_U)\) such that \(k\circ j_U\restriction M_D = j_D(j_U)\).
    We therefore have \(A\in U\cap M\) if and only if \(A\in P^M(X)\) and
    \([\text{id}]_U\in j_U(A)\), and this holds if and only if
    \(k([\text{id}]_U) \in j_D(j_U)(A)\). Clearly the set
    \[\{A\in P^{M_D}(X) : k([\text{id}]_U) \in j_D(j_U)(A)\}\]
    belongs to \(M_D\), since it is definable from parameters over \(M_D\).
    Therefore \(U\cap M_D\in M_D\).

    But \(U\cap M_D = W\). This completes the proof.
\end{proof}

\subsection{The approximation property}\label{AppxSection}
\begin{comment}
Continuing in this vein, we make the following definition:
\begin{defn}
Suppose \(\kappa\) is a cardinal, \(M\) is a model of set theory, and \(U\) is
an \(M\)-ultrafilter. We say \(W\) is {\it \(M\)-\(\kappa\)-complete} if for
every set \(\sigma\subseteq W\) with \(\sigma\in M\) and \(|\sigma| < \kappa\),
\(\bigcap \sigma \neq \emptyset\).
\end{defn}
\end{comment}
This section proves some basic structural results about pseudogrounds
under large cardinal assumptions.
We will show that if there is a proper class of
strongly compact cardinals, then the pseudogrounds are closed under
the fundamental model constructions of set theory: generic extensions
and extender ultrapowers.

For the sake of exposition, let us recall a theorem of Woodin
and Usuba that motivates the results of this section. This
requires some definitions.
\begin{defn}
An ultrafilter \(\mathcal U\) over a family \(F\) of subsets of \(X\) is {\it normal} 
if for any \(\langle A_x : x\in X\rangle\) with \(A_x\in \mathcal U\) for all \(x\in X\),
the {\it diagonal intersection} \[\triangle_{x\in X} A_x = \left\{\sigma\in F : \sigma\in \bigcap_{x\in \sigma} A_x\right\}\]
belongs to \(\mathcal U\).

A cardinal \(\kappa\) is {\it supercompact} if for all \(\lambda \geq \kappa\),
there is a \(\kappa\)-complete normal fine ultrafilter over \(P_\kappa(\lambda)\),
or equivalently, there is an elementary embedding \(j : V\to M\) where
 \(M\) is an inner model closed under \(\lambda\)-sequences.

An inner model \(M\) is a {\it weak extender model for the supercompactness of \(\kappa\)}
if for all \(\lambda\geq \kappa\), there is a \(\kappa\)-complete normal fine ultrafilter 
\(\mathcal U\) over \(P_\kappa(\lambda)\) such that \(P_\kappa(\lambda)\cap M\in \mathcal U\)
and \(\mathcal U\cap M\in M\).
\end{defn}

Woodin and Usuba independently proved the following theorem:
\begin{thm}\label{WoodinUsuba}
If \(M\) is a weak extender model for the supercompactness of \(\kappa\), then
\(M\) is a \(\kappa\)-pseudoground.
\end{thm}
Woodin asked whether the converse holds: if \(\kappa\) is supercompact,
must every \(\kappa\)-pseudo\-ground be a weak extender model for the 
supercompactness of \(\kappa\)?
The author found the following counterexample, based on Magidor's identity crisis \cite{Magidor}
as treated by Mitchell \cite{MitchellSkies}:
\begin{thm}\label{WEMCounterexample}
Suppose \(\kappa\) is strongly compact.
Then there is a \(\kappa\)-pseudoground in which
\(\kappa\) is the least measurable cardinal.
\end{thm}
We sketch the proof after \cref{AppxEquiv}.

This raises a natural question in the context of strongly compact cardinals.
Although \cref{WEMCounterexample} shows that a supercompact cardinal
need not be supercompact in a \(\kappa\)-pseudoground, a theorem of Hamkins \cite{HamkinsApprox}
shows that there is no corresponding counterexample for strong compactness.
Therefore one might hope that by considering
weak extender models for strong compactness, one might obtain a theorem like 
\cref{WoodinUsuba} that is an equivalence.
\begin{defn}
An inner model \(M\) is a {\it weak extender model for the strong compactness of \(\kappa\)}
if for every \(\lambda \geq \kappa\), there is a \(\kappa\)-complete fine ultrafilter \(\mathcal U\)
over \(P_\kappa(\lambda)\) such that \(P_\kappa(\lambda)\cap M\in \mathcal U\) and
\(\mathcal U\cap M\in M\).
\end{defn}

The question becomes whether every weak extender model for the strong compactness of \(\kappa\)
is a \(\kappa\)-pseudoground. The answer is again no:
\begin{prp}\label{SCWEM}
Suppose \(\kappa\) is strongly compact and \(M\) is an inner model.
Then \(M\) is a weak extender model for the strong compactness of \(\kappa\) if and only if
\(\kappa\) is strongly compact in \(M\) and \(M\) has the \(\kappa\)-cover property.
\begin{proof}
We begin with the forwards direction. 
To see that \(\kappa\) is strongly compact in \(M\), 
just note that if \(\mathcal U\) is a fine ultrafilter on \(P_\kappa(\lambda)\)
such that \(P_\kappa(\lambda)\cap M\in \mathcal U\) and \(\mathcal U\cap M\in M\),
then \(\mathcal U\cap M\) is a fine ultrafilter in \(M\).
To see that \(M\) has the \(\kappa\)-cover property, 
fix a cardinal \(\lambda\) and a \(\kappa\)-complete fine ultrafilter \(\mathcal U\)
on \(\lambda\) such that \(P_\kappa(\lambda)\cap M\in \mathcal U\) and 
\(\mathcal U\cap M\in M\). We will show that every \(\sigma\in P_\kappa(\lambda)\)
is contained in some \(\tau\in P_\kappa(\lambda)\cap M\). 
For any \(\sigma\in P_\kappa(\lambda)\)
\(\{\tau \in P_\kappa(\lambda) : \tau\subseteq \sigma\}\in \mathcal U\)
by fineness and \(\kappa\)-completeness. Therefore 
\(\{\tau \in P_\kappa(\lambda) : \tau\subseteq \sigma\}\cap M\in M\) since
\(P_\kappa(\lambda)\cap M\in \mathcal U\).
In other words, there is some \(\tau\in P_\kappa(\lambda)\cap M\) such that
\(\sigma\subseteq \tau\).

For the converse, fix \(\lambda \geq \kappa\).
Let \(\mathcal W\) be a \(\kappa\)-complete fine ultrafilter on \(P_\kappa(\lambda)\)
that belongs to \(M\). Then by the \(\kappa\)-cover property, \(\mathcal W\)
generates a \(\kappa\)-complete filter \(F\) in \(V\). Since \(\kappa\) is strongly compact,
\(F\) extends to a \(\kappa\)-complete ultrafilter \(\mathcal U\).
Now \(\mathcal U\) is a fine ultrafilter, \(P_\kappa(\lambda)\cap M\in \mathcal W\subseteq \mathcal U\),
and \(\mathcal U\cap M = \mathcal W\in M\). Since
\(\lambda\geq \kappa\) was arbitrary, this shows that \(M\) is a weak extender model
for the strong compactness of \(\kappa\).
\end{proof}
\end{prp}

\begin{cor}
Suppose \(\kappa\) is strongly compact and \(U\) is a \(\kappa^+\)-complete ultrafilter.
Then \(M_U\) is a weak extender model for the strong compactness of \(\kappa\),
but \(M_U\) does not have the \(\kappa\)-approximation property.
\begin{proof}
It is clear that \(M_U\) is a weak extender model for the strong compactness
of \(\kappa\) by \cref{SCWEM} since \(M_U\) is closed under \(\kappa\)-sequences. 
On the other hand, \(M_U\) does not have the 
\(\kappa\)-approximation property by the Laver-Hamkins uniqueness theorem
\cref{HamkinsUniqueness}, since \(H(\kappa^+)\cap M_U = H(\kappa^+)\) yet \(M_U\neq V\).
\end{proof}
\end{cor}

Despite \cref{SCWEM}, we will show that there is a variant 
of the notion of a weak extender model for strong compactness 
that coincides with the property of being a pseudoground.

\begin{thm}\label{AppxEquiv} Suppose \(\kappa\) is strongly compact and \(M\) is
    a model of set theory with the \(\kappa\)-cover property. Then the following
    are equivalent:
    \begin{enumerate}[(1)]
        \item \(M\) has the \(\kappa\)-approximation property.
        \item Every \(\kappa\)-complete ultrafilter over a set in \(M\)
            extends an ultrafilter of \(M\).
    \end{enumerate}
\end{thm}

For the proof, we need the concept of a {\it close embedding}.
This is a special case of a fine-structural
notion introduced by Mitchell-Steel \cite{MitchellSteel}.
Its utility in the coarse large cardinal setting was
first realized by Woodin \cite{Woodin}.
\begin{defn}
Suppose \(M\) and \(N\) are models of set theory. 
An elementary embedding \(j : M\to N\) is {\it close to \(M\)} if 
\(j\) is cofinal and every \(M\)-ultrafilter derived from \(j\) belongs to \(M\).
\end{defn}
The author noticed that closeness has a very simple
model theoretic characterization that simplifies a number of proofs.
\begin{thm}\label{CloseLemma}
    Suppose \(M\) and \(N\) are models of set theory and 
    \(j : M\to N\) is an elementary embedding. Then the following are equivalent:
    \begin{enumerate}[(1)]
        \item \(j\) is close to \(M\).
        \item For any \(A\in N\), \(j^{-1}[A]\in M\).
    \end{enumerate}
\begin{proof}
    {\it (1) implies (2):} Fix \(A\in N\). Since \(j\) is cofinal, 
    there is some \(X\in M\) such that \(N\) satisfies
    \(A\in j(X)\), and let \(U\) be the \(M\)-ultrafilter over \(X\) derived from \(j\) using
    \(A\). Let \(i : M\to P\) be the ultrapower embedding associated with \(U\)
    and let \(k : P\to N\) be the unique factor embedding such that \(k\circ i = j\)
    and \(k([\text{id}]_U = A\). Let \(\bar A = [\text{id}]_U\).
    Then since \(i\) is definable over \(M\), \(i^{-1}[\bar A]\in M\).
    But \(i^{-1}[\bar A] = (k\circ i)^{-1}[i(\bar A)] = j^{-1}[A]\).

    {\it (2) implies (1):} Fix \(X\in M\) and \(a\in N\) such that \(N\)
    satisfies that \(a\in j(X)\). We will show that the \(M\)-ultrafilter
    \(U\) over \(X\) derived from \(j\) using \(a\) belongs to \(N\).
    Indeed, let \(P\) be the principal ultrafilter over \(j(X)\) concentrated at \(a\),
    as computed in \(N\).
    Then \(P\in N\) and \(U = j^{-1}[P]\).

    To see that \(j\) is cofinal, fix \(a\in N\). Let \(\alpha\in \text{Ord}^N\) be 
    least such that \(N\) satisfies \(a\in V_\alpha\). 
    Let \(\bar \alpha = j^{-1}[\alpha]\). By (2), \(\bar\alpha\in M\),
    and so \(\bar \alpha\) is an ordinal of \(M\).
    But \(\alpha \leq j(\bar \alpha)\): otherwise 
    \(N\) satisfies that \(\bar \alpha\in j^{-1}[\alpha] = \bar \alpha\),
    which is impossible. Therefore
    \(N\) satisfies that \(a\in j(X)\) where \(X = (V_{\bar \alpha})^M\).
\end{proof}
\end{thm}
Notice that Woodin's lemma that close embeddings are closed under composition
is completely transparent given this characterization.

\begin{proof}[Proof of \cref{AppxEquiv}]
        {\it (1) implies (2):} This implication is due to Hamkins
        \cite{HamkinsApprox} and does not require that \(\kappa\) is
        strongly compact. If \(\kappa\) is a strong limit cardinal, we can use
        \cref{ContinuumComplete} to obtain the stronger theorem
        that every \(M\)-\(\kappa\)-complete \(M\)-ultrafilter belongs to \(M\).
        (This seems to be a new result.) 
        
        Assume \(M\) has the \(\kappa\)-approximation
        property. Let \(U\) be a \(\kappa\)-complete ultrafilter over \(X\in
        M\). We must show that \(U\cap M\in M\). It suffices to show that
        \(U\cap M\) is \(\kappa\)-approximated by \(M\). Suppose
        \(\sigma\subseteq P^M(X)\) such that \(|\sigma| < \kappa\). We want to
        show that \(U\cap \sigma\in M\). Clearly it suffices to prove this in
        the case that \(\sigma\) is closed under relative complements in \(X\).
        By the \(\kappa\)-completeness of \(U\), \(\bigcap (U\cap \sigma)\neq
        \emptyset\), so fix \(x\in \bigcap (U\cap \sigma)\). Then since
        \(\sigma\) is closed under complements, \(U\cap \sigma = \{A\in\sigma :
        x\in A\}.\) It follows that \(U\cap \sigma\in M\).

        {\it (2) implies (1):} Suppose \(X\in M\) and \(A\subseteq X\) is
        \(\kappa\)-approximated by \(M\). Let \(j : V\to N\) be an elementary
        embedding with critical point \(\kappa\) such that \(j[X]\) is contained in
        a set \(B\in N\) with \(|B|^N < j(\kappa)\). By replacing \(B\) with
        \(B\cap j(M)\), we may assume without loss of generality that
        \(B\subseteq j(M)\). Since \(j(M)\) has the \(j(\kappa)\)-cover property
        in \(N\), there is some \(C\in j(M)\) with \(|C|^N < j(\kappa)\) such
        that \(B\subseteq C\). Since \(j(A)\) is \(j(\kappa)\)-approximated by
        \(j(M)\), \(j(A)\cap C\in j(M)\). We have assumed that \(U\cap M\in M\)
        for every \(\kappa\)-complete ultrafilter \(U\) over a set in \(M\);
        therefore for every \(M\)-ultrafilter \(W\)
        derived from \(j\), \(W\in M\). In other words, \(j\) is close to \(M\).
        Therefore by \cref{CloseLemma}, \(j^{-1}[j(A)\cap C] \in M\). But
        \(j^{-1}[j(A)\cap C] = A\), so \(A\in M\), as desired.
\end{proof}

Using \cref{AppxEquiv}, we prove \cref{WEMCounterexample}.
\begin{proof}[Sketch of \cref{WEMCounterexample}]
    Let \(\text{Meas}\) denote the class of measurable cardinals. 
    Choose \(\vec U = \langle U_\delta : \delta \in \text{Meas}\cap \kappa\rangle\) such that 
    \(U_\delta\) is a normal ultrafilter on \(\delta\) with
    \(\text{Meas}\cap \delta\notin U_\delta\).
    We define a sequence of ordinals \(\langle \delta_\alpha : \alpha < \kappa\rangle\)
    and an iterated ultrapower \[\langle M_\alpha, W_\beta, j_{\beta,\alpha} : \beta < \alpha \leq \kappa\rangle\]
    by simultaneous recursion, letting \(\delta_\alpha\) 
    be the least measurable cardinal \(\delta\) of \(M_\alpha\) such that
    the set of preimages \(\{\beta < \alpha : j_{\beta,\alpha}(\delta_\beta) = \delta\}\) of \(\delta\) is finite.
    Let \(W_\alpha = j_{0\alpha}(\vec U)_{\delta_\alpha}\). The rest of the data of the iterated ultrapower is
    uniquely determined by the sequence \(\langle W_\alpha : \alpha < \kappa\rangle\) in the usual way.
    
    For \(\gamma \leq \kappa\), let \(G_\gamma = \{\delta_\alpha : \alpha < \gamma\}.\)
    One can show that \(\kappa\) is the least
    measurable cardinal of \(M_\kappa[G_\kappa]\). The proof appears in
    \cite[Theorem 1.2]{MitchellSkies}. 
    
    It is easy to see that \(M_\kappa\) has the \(\kappa\)-cover property:
    this follows from the fact that \(j_{0\kappa} = \kappa\) and the proof of \cref{UltraCov}.
    Since \(M_\kappa\subseteq M_\kappa[G]\), \(M_\kappa[G]\) has the \(\kappa\)-cover property as well.
    
    To finish we must show that \(M_\kappa[G]\) has the \(\kappa\)-approximation property.
    By \cref{AppxEquiv}, it suffices to show that for any
    elementary embedding \(i : V\to N\) such that 
    \(N^\kappa\subseteq N\) and \(\text{crit}(i) \geq \kappa\),
    \(i\restriction M_\kappa\) is amenable to \(M_\kappa\).
    The proof is due to Mitchell and appears in \cite[Theorem 1.2]{MitchellSkies}.
    We have 
    \[i(j_{0\kappa}) = (j_{0,i(\kappa)})^N = j_{\kappa,i(\kappa)}^N \circ j_{0\kappa}\restriction N\]
    The final equality uses the \(\kappa\)-closure of
    \(N\) to deduce that \((j_{0\kappa})^N = j_{0\kappa}\restriction N\).
    We claim
    \[i\restriction M_\kappa = j_{\kappa,i(\kappa)}^N \circ j_{0\kappa}(i)\]
    We have 
    \[i\circ j_{0\kappa} = i(j_{0\kappa}) \circ i = j_{\kappa,i(\kappa)}^N \circ j_{0\kappa}\circ i
    = j_{\kappa,i(\kappa)}^N \circ j_{0\kappa}(i) \circ j_{0\kappa}\]
    In other words, \(i\) and \(j_{\kappa,i(\kappa)}^N \circ j_{0\kappa}(i)\)
    agree on \(j_{0\kappa}[V]\).
    They also agree on \(\kappa\), since both embeddings are the identity on \(\kappa\).
    Since \(M_\kappa = H^{M_\kappa}(j_{0\kappa}[V]\cup \kappa)\), it follows that 
    \(i = j_{\kappa,i(\kappa)}^N \circ j_{0\kappa}(i)\), as claimed. As a consequence,
    \(i\) is amenable to \(M_\kappa\).
    
    Note that \(i(G) = G\cup H\) where \(H\) is the sequence of indiscernibles
    generated by \(j_{\kappa,i(\kappa)}^N\) in much the same way that \(G\) is generated from
    \(j_{0\kappa}\). Therefore \(i(G) \in M_\kappa[G]\).
    It follows that \(i\) is amenable to \(M_\kappa[G]\): every element of \(M_\kappa[G]\)
    is \(\Sigma_2\)-definable from parameters in 
    \(M_\kappa\cup \{G\}\), and \(i\restriction M_\kappa\cup \{G\}\) 
    is amenable to \(M_\kappa[G]\).

    In particular, it follows that for every \(\kappa\)-complete ultrafilter
    \(U\) over a set \(X\) in \(M\), \(j_U\restriction M\) is amenable to \(M\),
    and therefore \(U\cap M\in M\), since \(U\cap M\) is the \(M\)-ultrafilter
    over \(X\) derived from \(j_U\restriction M\) using \([\text{id}]_U\).
    Applying \cref{AppxEquiv}, it follows that \(M\) has the \(\kappa\)-approximation property.
    \end{proof}
Combining \cref{AmenableThm}, \cref{UltraCov}, and \cref{AppxEquiv} immediately
yields that small ultrapowers are pseudogrounds:
\begin{cor}\label{UltraAppx} Suppose \(\kappa\) is a strongly compact cardinal
and \(D\in V_\kappa\) is an ultrafilter. Then \(M_D\) has the
\(\kappa\)-approximation property.\qed
\end{cor}
In fact, by generalizing the Commuting Ultrapowers Lemma (\cref{CompleteCommute})
to work in the case where one embedding is an external extender,
one can prove a much stronger result:
\begin{cor}
    Suppose \(\kappa\) is a strongly compact cardinal.
    Suppose \(N\) is \(\kappa\)-pseudoground and 
    \(E\) is an \(N\)-extender in \(V_\kappa\).
    Then \(M_E^N\) is a \(\kappa\)-pseudoground.\qed
\end{cor}

It is natural to ask whether one can drop the cover assumption in \cref{AppxEquiv}.
Suppose there is a proper class of extendible
cardinals and \(M\) is an inner model 
such that every sufficiently complete ultrafilter extends an ultrafilter of \(M\).
Must \(M\) be a pseudoground? 
The answer is probably no, but
our next theorem reaches in this direction:
\begin{thm}\label{CKappa}
Suppose \(\kappa\) is strongly compact and \(M\) is an inner model. Then the
following are equivalent:
\begin{enumerate}[(1)]
\item \(M\) is a \(\kappa\)-pseudoground.
\item \(\kappa\) is strongly compact in \(M\) and the following hold:
    \begin{itemize}
        \item Every \(\kappa\)-complete ultrafilter
        over a set in \(M\) extends an ultrafilter of \(M\).
        \item Every \(\kappa\)-complete ultrafilter of \(M\)
        extends to a \(\kappa\)-complete ultrafilter of \(V\).
    \end{itemize}
\item Every regular cardinal of \(M\) above \(\kappa\) has cofinality
    at least \(\kappa\) and every \(\kappa\)-complete ultrafilter
    over a set in \(M\) extends an ultrafilter of \(M\).
\end{enumerate}
\begin{proof}
{\it (1) implies (2)}. 
The first bullet is just the theorem of Hamkins proved in \cref{AppxEquiv}.
The second bullet uses \cref{ContinuumComplete} to conclude that
\(M\)-\(\kappa\)-complete ultrafilters generate \(\kappa\)-complete filters. 

The fact that \(\kappa\) is strongly compact in \(M\) is also due to Hamkins.
We give a different proof.
By the \(\kappa\)-cover property, the \(M\)-filter over \(P_\kappa(\lambda)\cap M\)
generated by sets of the form 
\(\{\sigma \in P_\kappa(\lambda)\cap M : \alpha\in \sigma\}\)
generates a \(\kappa\)-complete filter, and therefore extends to a \(\kappa\)-complete
ultrafilter of \(V\), which in turn extends a \(\kappa\)-complete ultrafilter
\(\mathcal W\) of \(M\); \(\mathcal W\) is fine and therefore witnesses that \(\kappa\) is
\(\lambda\)-strongly compact in \(M\).

{\it (2) implies (3):} Fix an \(M\)-regular cardinal \(\delta\) above \(\kappa\).
Since \(\kappa\) is strongly compact in \(M\), \(M\) satisfies that there is a 
\(\kappa\)-complete uniform ultrafilter \(U\) on \(\delta\). Since \(U\) 
extends to a \(\kappa\)-complete ultrafilter, \(\delta\) must have
cofinality at least \(\kappa\): indeed, any ordinal that carries a \(\kappa\)-complete
fine ultrafilter is necessarily of cofinality at least \(\kappa\).

{\it (3) implies (1):}
By \cref{AppxEquiv}, it suffices to prove that \(M\) has the \(\kappa\)-cover property.
Fix a regular uncountable cardinal \(\delta\).
We will find a set \(X\in M\) with \(|X|^M = \delta\)
and a \(\kappa\)-complete
fine ultrafilter over \(P_\kappa(X)\) such that \(P_\kappa(X)\cap M\in \mathcal U\).
It then follows immediately that \(M\) has the \(\kappa\)-cover property at \(X\):
indeed, if \(\sigma\in P_\kappa(X)\), then 
\(\{\tau\in P_\kappa(X) : \sigma\subseteq \tau\}\in \mathcal U\) since \(\mathcal U\) is
\(\kappa\)-complete and fine,
so \(\{\tau \in P_\kappa(X) : \sigma\subseteq \tau\}\cap M\) is nonempty
since \(P_\kappa(X)\cap M\in \mathcal U\).
In other words, there is some \(\tau\in P_\kappa(X)\cap M\) 
with \(\sigma\subseteq \tau\). 
This proves that \(M\) has the \(\kappa\)-cover property at \(X\).
Since we can make \(X\) arbitrarily large,
it follows that \(M\) has the \(\kappa\)-cover property.

Let \(j : V\to N\) be an elementary embedding with critical point \(\kappa\)
such that for some 
\(B\in N\) with \(|B|^N < j(\kappa)\), \(j[\delta]\subseteq B\).
In particular, \(N\) satisfies that \(\text{cf}(\sup j[\delta]) < j(\kappa)\).
It follows that \(j(M)\) satisfies that \(\text{cf}(\sup j[\delta]) < j(\kappa)\): 
if \(\text{cf}^{j(M)}(\sup j[\delta]) \geq j(\kappa)\),
then by our assumption shifted to \(j(M)\),
its cofinality in \(N\) is at least \(j(\kappa)\), a contradiction.
Therefore we may fix a closed unbounded set \(C\subseteq \sup j[\delta]\)
such that \(C\in j(M)\) and \(|C|^{j(M)} < j(\kappa)\).

Since every \(\kappa\)-complete ultrafilter over a set in \(M\)
extends an ultrafilter of \(M\),
the embedding \(j\) is close to \(M\). Therefore by \cref{CloseLemma},
since \(C\in j(M)\), \(j^{-1}[C]\in M\).

Let \(X = j^{-1}[C]\). Since \(j[\delta]\) and \(C\)
are \(\omega\)-closed unbounded subsets of the ordinal \(\sup j[\delta]\)
(which has uncountable cofinality), \(j[\delta]\cap C\) is unbounded in \(\sup j[\delta]\).
But \(j[X] = j[\delta]\cap C\), so \(X\) must also be unbounded in \(\delta\).
Therefore since \(\delta\) is regular, \(|X| = \delta\).

Let \(D = C\cap j(X)\). Then \(D\in j(M)\), \(j[X]\subseteq D\),
\(D \subseteq j(X)\), and \(|D|^{j(M)} < j(\kappa)\).
In particular, \(D\in j(P_\kappa(X))\),
so it makes sense to derive an ultrafilter \(\mathcal U\) over 
\(P_\kappa(X)\) from \(j\) using \(D\).
Since \(j\) has critical point \(\kappa\), \(\mathcal U\) is \(\kappa\)-complete.
Since \(j[X]\subseteq D\), \(\mathcal U\) is fine. Since \(D\in j(M)\),
\(P_\kappa(X)\cap M\in \mathcal U\) by the definition of a derived ultrafilter.

Thus we have found a set \(X\) of cardinality \(\delta\) and a \(\kappa\)-complete fine 
ultrafilter \(\mathcal U\) over \(P_\kappa(X)\) such that \(P_\kappa(X)\cap M\in \mathcal U\).
This completes the proof.
\end{proof}
\end{thm}
Of course, \cref{CKappa} (3) is equivalent to the statement that
ultrafilters extend ultrafilters of \(M\) and \(M\) correctly computes
the class of cardinals of cofinality less than \(\kappa\).
\begin{cor}
Suppose \(\kappa\) is strongly compact and \(M\) is an inner model.
Assume \(M\) is cardinal correct and
every \(\kappa\)-complete ultrafilter over a set in \(M\)
extends an ultrafilter of \(M\).
Then \(M\) is a \(\kappa\)-pseudoground.
\begin{proof}
We will show that any \(M\)-regular cardinal \(\delta\) greater than or equal to \(\kappa\)
has cofinality at least \(\kappa\), which by \cref{CKappa} implies the corollary. 

A theorem of Viale \cite[Theorem 27]{Viale} states that if \(\kappa\) is strongly compact,
\(M\) is an inner model, and \(\lambda\geq \kappa\) is an \(M\)-regular cardinal
such that \(\lambda^{+M} = \lambda^+\), then \(\text{cf}(\lambda)\geq \kappa\).

Now fix an \(M\)-regular cardinal \(\delta\geq\kappa\).
Since \(\delta^{+M} = \delta^+\),
\(\text{cf}(\delta) \geq \kappa\) by Viale's theorem. 
\end{proof}
\end{cor}

An immediate consequence of \cref{AppxEquiv} is
the transitivity of the pseudoground order:
\begin{cor}\label{Transitive}
If \(\kappa\) is a strongly compact cardinal, then a
\(\kappa\)-pseudoground of a \(\kappa\)-pseudo\-ground is a \(\kappa\)-pseudoground.\qed
\end{cor}

\begin{cor}\label{GenericAppx}
Suppose \(\kappa\) is a strongly compact cardinal and \(M\) 
is a \(\kappa\)-pseudoground of \(V\).
Assume \(\mathbb P \in V_\kappa\cap M\) is a partial order and \(G\subseteq
\mathbb P\) is an \(M\)-generic filter that belongs to \(V\). Then
\(M[G]\) is a \(\kappa\)-pseudoground of \(V\).
\begin{proof}
Clearly \(M[G]\) has the \(\kappa\)-cover property. 
We will show that \(M[G]\) inherits \(\kappa\)-complete ultrafilters.
It easily suffices to show that \(M[G]\) inherits \(\kappa\)-complete ultrafilters
over sets in \(M\).
Suppose \(U\) is a \(\kappa\)-complete ultrafilter over a set in \(M\). 
We must show that \(U\cap M[G]\in M[G]\). Note
that \(U\cap M\in M\) by \cref{AppxEquiv}. By the L\'evy-Solovay theorem, the
filter \(W\) generated by \(U\cap M\) in \(M[G]\) is an \(M[G]\)-ultrafilter.
Clearly \(W\in M[G]\) since \(U\cap M\in M[G]\). But since \(W\subseteq U\cap
M[G]\), in fact \(W = U\cap M[G]\) (by the maximality of ultrafilters).
\end{proof}
\end{cor}

It is tempting to doubt that this corollary really requires
a strongly compact cardinal,
but the following fact comes close to showing that this hypothesis
is optimal:
\begin{prp}
Suppose \(\kappa\leq \lambda\) are regular cardinals and \(\kappa\) is \({<}\lambda\)-strongly compact. 
Then there is a cofinality preserving forcing extension \(N\) of \(V\) with the following properties:
\begin{itemize}
    \item \(\kappa\) remains \({<}\lambda\)-strongly compact in \(N\). 
    \item \(V\) is an \(\omega_1\)-pseudoground of \(N\).
    \item For some \(V\)-generic Cohen real \(g\in N\),
\(V[g]\) does not have the \(\lambda\)-approximation property in \(N\).
\end{itemize}
\begin{proof}
    Let \(g\) be a \(V\)-generic Cohen real and \(G\subseteq (\text{Add}(\lambda,1))^{V[g]}\) be 
    \(V[g]\)-generic. Let \(N = V[g][G]\). Then \(V\) has the 
    \(\omega_1\)-approximation and cover properties in \(N\)
    by a theorem of Hamkins \cite[Lemma 13]{Hamkins}, and obviously \(V[g]\)
    does not have the \(\lambda\)-approximation property in \(N\), since 
    by the \(\lambda\)-closure of \((\text{Add}(\lambda,1))^{V[g]}\), 
    \(N\) satisfies that \(G\) is \(\lambda\)-approximated by \(V[g]\),
    and yet \(G\notin V[g]\) by genericity.

    The L\'evy-Solovay theorem implies that \(\kappa\) remains \({<}\lambda\)-strongly compact in \(V[g]\).
    The \(\lambda\)-closure of \((\text{Add}(\lambda,1))^{V[g]}\) therefore implies that
    \(\kappa\) remains \({<}\lambda\)-strongly compact in \(N\):
    every \(N\)-regular cardinal \(\delta\)
    with \(\kappa \leq \delta < \lambda\)
    carries a uniform \(\kappa\)-complete ultrafilter in \(V[g]\),
    which is in fact an \(N\)-ultrafilter since \(P^N(\delta) = P^{V[g]}(\delta)\).
\end{proof}
\end{prp}
\section{Silver's question}\label{SilverSection}
\subsection{Indecomposable ultrafilters}
Let us reintroduce some concepts defined in the introduction.
\begin{defn}
Suppose \(\delta < \lambda\) are cardinals.
An ultrafilter \(U\) over a set \(X\) is {\it \((\delta,\lambda)\)-indecomposable}
if any partition \(\langle A_\nu\rangle_{\nu < \alpha}\) of \(X\) with
\(\alpha < \lambda\) has a subsequence \(\langle A_{\nu_\xi}\rangle_{\xi < \beta}\)
with \(\beta < \delta\) whose union belongs to \(U\).
\end{defn}

The following combinatorial characterization of indecomposability is sometimes convenient.
\begin{defn}
Suppose \(U\) is an ultrafilter over a set \(X\) and \(\gamma\) is a cardinal. A
{\it \(\gamma\)-decomposition of \(U\)} is a function \(f : X\to \gamma\) such
that for any \(A\in U\), \(f[A]\) has cardinality \(\gamma\); \(U\) is {\it
\(\gamma\)-decomposable} if it has a \(\gamma\)-decomposition, and \(U\) is {\it
\(\gamma\)-indecomposable} otherwise.
\end{defn}
Thus \(U\) is \(\gamma\)-indecomposable if and only if \(U\) is
\((\gamma,\gamma^+)\)-indecomposable and \(U\) is \((\delta,\lambda)\)-indecomposable
if and only if \(U\) is not \(\gamma\)-decomposable for any cardinal
\(\gamma\) such that \(\delta \leq \gamma < \lambda\).

The concept of a \(\gamma\)-decomposition is best understood in terms of 
pushforwards (\cref{PushDef}).
Notice that \(f : X \to \gamma\) is a \(\gamma\)-decomposition of \(U\) if and
only if the pushforward \(f_*(U)\) is a uniform ultrafilter over \(\gamma\).
Thus \(U\) is \(\gamma\)-decomposable if and only if \(U\) pushes forward to a
uniform ultrafilter over \(\gamma\).

We remark that if \(\gamma\) is regular, then
an ultrafilter \(U\) is \(\gamma\)-indecomposable if and only if
\(U\) is closed under intersections of descending sequences of sets of length
\(\gamma\).

The following theorem, due to Silver, is a key element in all of our
applications:
\begin{thm}[Silver]\label{Silver} Suppose \(\delta\) and \(\kappa\) are
    cardinals with \(2^\delta < \kappa\). Suppose \(U\) is a
    \((\delta,\kappa)\)-indecomposable ultrafilter over a set \(X\). Then \(j_U
    = j_W^{M_D}\circ j_D\) where \(D\) is an ultrafilter over a cardinal less
    than \(\delta\) and \(W\) is an 
    \(M_D\)-\(\kappa\)-complete \(M_D\)-ultrafilter over \(j_U(X)\).\qed
\end{thm}
Silver \cite[Lemma 2]{Silver} sketches a proof in the case that \(\delta =
\omega_1\). The author's thesis \cite[Theorem 7.5.24]{UA} contains a more detailed proof, 
assuming for superficial reasons that \(U\) is countably complete. 
\subsection{Silver's question above a strongly compact}
As an immediate consequence of \cref{AmenableThm} and Silver's factorization
theorem (\cref{Silver}), we obtain a factorization theorem for indecomposable
ultrafilters:
\begin{thm}\label{InternalDecomp} Suppose \(\delta < \kappa\) are
cardinals and \(X\) is a set. Assume \(\kappa\) is a strong limit cardinal and every
\(\kappa\)-complete filter over \(X\) extends to a \(\delta\)-complete
ultrafilter. 
Suppose \(U\) is a \((\delta,\kappa)\)-indecomposable ultrafilter over \(X\). 
Then \[j_U = (j_W)^{M_D}\circ j_D\] where \(D\) is an
ultrafilter over a cardinal less than \(\delta\) and \(W\) is a \(\kappa\)-complete
ultrafilter of \(M_D\) over \(j_D(X)\).
\begin{proof}
Applying Silver's Theorem (\cref{Silver}) to \(U\) yields that \(j_U =
j_W^{M_D}\circ j_D\) where \(D\) is an ultrafilter over a cardinal less than
\(\delta\) and \(W\) is an \(M_D\)-ultrafilter over \(j_D(\lambda)\) that is
\(M_D\)-\(\kappa\)-complete. By
\cref{AmenableThm}, \(W\in M_D\). Moreover by \cref{ContinuumComplete}
and \cref{UltraCov}, \(W\) is \(\kappa\)-complete.
\end{proof}
\end{thm}

\begin{thm}\label{IndecompChar0} 
    Suppose \(\delta < \kappa \leq \lambda\) are
    cardinals. Assume \(\kappa\) is a strong limit cardinal and every
    \(\kappa\)-complete filter over \(\lambda\) extends to a \(\delta\)-complete
    ultrafilter. 
    Suppose there is a \((\delta,\lambda)\)-indecomposable ultrafilter over \(\lambda\).
    Then \(\lambda\) either is a measurable cardinal or 
    \(\lambda\) has cofinality less than \(\delta\) and \(\lambda\) is a limit of measurable cardinals.
    \begin{proof}
        Applying \cref{InternalDecomp},
         \[j_U = (j_W)^{M_D}\circ j_D\] where \(D\) is an
        ultrafilter over a cardinal less than \(\delta\) and \(W\) is a \(\kappa\)-complete
        ultrafilter of \(M_D\) over \(j_D(\lambda)\).

        Since \(U\) is \((\delta,\lambda)\)-indecomposable, \(W\) must be
        \((j_D(\delta),j_D(\gamma))\)-indecompos\-able in \(M_D\) for all cardinals \(\gamma < \lambda\).
        To see this, fix an \(M_D\)-cardinal
        \(\eta\) and an \(\eta\)-decomposition 
        \(f : j_D(\lambda)\to \eta\) of \(W\) in \(M_D\). 
        Let \(\gamma\) be the least cardinal such that \(j_D(\gamma)\geq \eta\),
        and assume \(\gamma < \lambda\).
        We must show that
        \(\eta < j_D(\delta)\). Let \(g : \lambda\to \gamma\) be such that
        \([g]_U = [f]_W^{M_D}\). Since \(W\) is \(\gamma\)-indecomposable,
        there is a set \(A\in U\) such that \(|g[A]| < \gamma\).

        Since \(\gamma\) is least such that \(j_D(\gamma) \geq \eta\),
        \(j_D(|g[A]|)< \eta\).
        Since \[[f]_W^{M_D} = j_U(g)([\text{id}]_U)\in j_U(g[A]) = j_W^{M_D}(j_D(g[A]))\]
        \L o\'s's Theorem implies that 
        \(f(\alpha) \in j_D(g[A])\) for all \(\alpha\) in a \(W\)-large set \(B\).
        In other words, \(f[B] \subseteq j_D(g[A])\), and hence
        \(|f[B]|^{M_D} \leq j_D(|g[A]|) < \eta\).
        This contradicts that \(f\) is an \(\eta\)-decomposition in \(M_D\).

        Let \(\rho\) be the completeness of \(W\) as computed in \(M_D\),
        the least \(M_D\)-cardinal such that \(W\) is not \(\rho\)-complete in \(M_D\).
        The completeness of a countably complete ultrafilter is always a measurable cardinal,
        and so since \(W\) is \(\kappa\)-complete, 
        \(\rho\geq \kappa\) and \(\rho\) is measurable.
        Moreover, \(W\) is \(\rho\)-decomposable, so
        since \(W\) is \((j_D(\delta),j_D(\gamma))\)-indecomposable
        for all \(\gamma < \lambda\), \(\rho\) must be greater than \(j_D(\gamma)\)
        for all \(\gamma < \lambda\). On the other hand, since \(W\) is 
        an \(M_D\)-ultrafilter over \(\lambda\), \(W\) is \(\gamma\)-indecomposable for all
        \(M_D\)-cardinals greater than \(\lambda\), and hence \(\rho \leq j_D(\lambda)\).

        Assume first that \(\lambda\) has cofinality at least 
        \(\delta\). Then \(j_D(\lambda) = \sup j_D[\lambda]\).
        This is a general fact; to see it let \(\iota\) be the underlying set of \(D\), 
        and note that any \(\alpha < j_D(\lambda)\) is equal to \([f]_U\) for some function
        \(f : \iota\to \lambda\), and \(\text{ran}(f)\subseteq \beta\) for some \(\beta < \lambda\)
        since \(\iota < \delta\leq \text{cf}(\lambda)\), so \(\alpha < j_D(\beta)\). 
        Since \(j_D(\lambda) = \sup j_D[\lambda]\) and \(j_D(\gamma) < \rho \leq j_D(\lambda)\)
        for all \(\gamma < \lambda\), \(\rho = j_D(\lambda)\).
        Therefore \(j_D(\lambda)\) is measurable in \(M_D\),
        and so \(\lambda\) is measurable by elementarity.
        This proves the theorem in the case that \(\lambda\) has cofinality at least \(\delta\).

        Otherwise, \(\lambda\) has cofinality less than \(\delta\).
        We finish by proving that in this case \(\lambda\) is a limit
        of measurable cardinals. Since \(\rho\) is regular in \(M_D\), \(\rho \neq j_D(\lambda)\).
        Therefore \(\rho < j_D(\lambda)\).  This is a standard reflection argument. 
        Suppose \(\gamma < \lambda\). We will show that there is a measurable cardinal
        between \(\gamma\) and \(\lambda\).
        Of course, \(M_D\) satisfies that there is a measurable cardinal between \(j_D(\gamma)\)
        and \(j_D(\lambda)\), namely \(\rho\). Therefore by elementarity,
        there is a measurable cardinal between \(\gamma\) and \(\lambda\), as desired. 
    \end{proof}
\end{thm}

We state a special case which answers Silver's question above the least strongly
compact cardinal:
\begin{thm}Suppose \(\lambda\) is greater than or equal to
    the least strongly compact cardinal and carries an indecomposable
    ultrafilter. Either \(\lambda\) is measurable or else
    \(\textnormal{cf}(\lambda) = \omega\) and \(\lambda\) is a limit of
    measurable cardinals.\qed
\end{thm}

From \cref{InternalDecomp},  one can extract a topological 
characterization of indecomposable
ultrafilters above the least strongly compact cardinal. Recall that
\(\beta(\lambda)\) denotes the space of 
ultrafilters on \(\lambda\) with the Stone-\v Cech topology.
\begin{thm}\label{IndecompChar1} Suppose \(\lambda\) is greater than or equal to
    the least strongly compact cardinal and \(U\) is an ultrafilter over \(\lambda\). 
    Then the following are equivalent:
    \begin{itemize}
        \item \(U\) is indecomposable.
        \item Either \(U\) is \(\lambda\)-complete or
            \(U\) lies in the closure of a countable discrete
            set of ultrafilters \(S\subseteq \beta(\lambda)\) such that
            for any \(\gamma < \lambda\), all but finitely many
            ultrafilters in \(S\) are \(\gamma^+\)-complete.\qed
    \end{itemize}
\end{thm}

\section{Almost strong compactness}\label{AlmostSection}
The principles of Bagaria-Magidor laid out in 
\cref{CompactnessSection} offer a spectrum of strong compactness 
properties. In a perfect world (for example, assuming UA), these
would be characterized in terms of the classical notion of strong compactness. 
But given Bagaria-Magidor's theorem \cite{MagidorBagaria}
that it is consistent with ZFC that the first \(\omega_1\)-strongly compact
cardinal is singular, it is natural to wonder whether there are any nontrivial 
relationships between these principles at all. The results of this section
show that there are subtle implications between 
classical strong compactness and Bagaria-Magidor's notion of almost strong compactness.
\subsection{Decomposability spectra}
Our results in this section make use of the observation
that assuming the Singular Cardinals Hypothesis, countably complete ultrafilters have
very simple {\it decomposability spectra,} a concept first studied by Lipparini.
\begin{defn}\label{KDDef} If \(U\) is an ultrafilter, the {\it decomposability
    spectrum of \(U\)}, denoted \(K_U\), is the set of all cardinals \(\lambda\)
    such that \(U\) is \(\lambda\)-decomposable.
\end{defn}

We use the following theorem of Lipparini:
\begin{thm}[Lipparini]\label{ProductDecomp} Suppose \(U\) is an ultrafilter and
\((\lambda_\alpha)_{\alpha < \eta}\) is an increasing sequence of infinite
cardinals in \(K_D\). Then there is a cardinal \(\delta\in K_D\) with
\(\sup_{\alpha < \eta} \lambda_{\alpha}\leq \delta \leq \prod_{\alpha < \eta}
\lambda_{\alpha}\).
\begin{proof}
For each \(\alpha < \eta\), choose a \(\lambda_\alpha\)-decomposition
\(f_\alpha\) of \(U\). Thus \(f_\alpha\) is a function from the underlying set
\(X\) of \(U\) to \(\lambda_\alpha\). Define \(f : X\to \prod_{\alpha < \eta}
\lambda_\alpha\) by \(f(x) = (f_\alpha(x))_{\alpha < \eta}\).

Fix \(A\in U\) such that  
\(|f[A]| = \delta\) is as small as possible. Note that \(\delta \leq
\prod_{\alpha < \eta} \lambda_{\alpha}\). Moreover, for all \(\alpha < \eta\),
\(\delta \geq \lambda_\alpha\): for \(A\in U\), \(|f[A]| \geq |f_\alpha[A]|\geq
\lambda_\alpha\) since \(f_\alpha\) is a \(\lambda_\alpha\)-decomposition.

Let \(p : \prod_{\alpha < \eta} \lambda_\alpha\to \delta\) be injective on \(A\)
and \(0\) on the complement of \(A\). Then \(p\circ f\) is a
\(\delta\)-decomposition of \(U\).
\end{proof}
\end{thm}

\begin{lma}[SCH]\label{SimpleSpectrum} Suppose \(U\) is a countably complete
ultrafilter. Suppose \(K_U\) is unbounded below a limit cardinal \(\lambda\).
Then all sufficiently large regular cardinals less than \(\lambda\) belong to
\(K_U\).
\end{lma}

For this, we need a well-known fact, a special case of a more general theorem of
Ketonen \cite{Ketonen}. For one approach, see \cite[Theorem 7.2.12]{UA}.
\begin{lma}\label{KetonenSpectrum} Suppose \(U\) is a \(\theta^+\)-complete
ultrafilter, \(\lambda\) is a singular cardinal of cofinality \(\theta\), and
\(\lambda^+\in K_U\). Then all sufficiently large regular cardinals below
\(\lambda\) are in \(K_U\).\qed
\end{lma}

\begin{proof}[Proof of \cref{SimpleSpectrum}]
We first handle the case in which \(\lambda\) has countable cofinality. Assume
towards a contradiction that the lemma fails. Let \((\lambda_n)_{n < \omega}\)
be a sequence of cardinals unbounded in \(\lambda\) such that \(\lambda_n \notin
K_U\) for all \(n < \omega\). By \cref{ProductDecomp} there is some \(\delta\in
K_U\) with \(\sup_{n < \omega} \lambda_n \leq \delta \leq \prod_{n < \omega}
\lambda_n\). Since \(U\) is countably complete, \(\delta \neq \sup_{n < \omega}
\lambda_n\). Note however that \(\prod_{n < \omega} \lambda_n \leq
\lambda^\omega = \lambda^+\) by SCH. Therefore \(\delta = \lambda^+\). Since
\(U\) is countably complete, \(\lambda\) has countable cofinality, and
\(\lambda^+\in K_U\), \cref{KetonenSpectrum} implies that all sufficiently large
regular cardinals less than \(\lambda\) belong to \(K_U\).

Now we take on the case that \(\lambda\) has uncountable cofinality. Let \(S\)
be the set of limit points of \(K_U\) of countable cofinality. Then \(S\) is
\(\omega\)-closed unbounded in \(\lambda\). Define \(f : S\to \lambda\) by
setting \(f(\alpha)\) equal to the least \(\gamma < \alpha\) such that every
regular cardinal between \(\gamma\) and \(\alpha\) belongs to \(K_U\). Note that
there is such a cardinal \(\gamma\) by the previous case. The function \(f\) is
nondecreasing and regressive, so there is some \(\gamma < \kappa\) such that
\(f(\alpha) =\gamma\) for all sufficiently large \(\alpha < S\). In other words,
every regular cardinal between \(\gamma\) and \(\kappa\) belongs to \(K_U\), as
desired.
\end{proof}

\begin{comment}
\begin{lma}[SCH]
Suppose \(U\) is a countably complete ultrafilter. Then \(K_U\cap
\textnormal{Reg}\) is the union of finitely many intervals of regular cardinals.
\begin{proof}
Suppose not, towards a contradiction. Let \(\lambda_0\) be the least element of
\(K_U\) and let \(\delta_0\) be the least regular cardinal above \(\lambda_0\)
that is not in \(K_U\). Proceeding by recursion, for \(n < \omega\), let
\(\lambda_{n+1}\) be the least regular cardinal above \(\delta_n\) in \(K_U\)
and let \(\delta_{n+1}\) be the least regular cardinal above \(\lambda_{n+1}\)
that is not in \(K_U\). Our assumption that \(K_U\) is not the union of finitely
many half-open intervals implies that \(\lambda_{n+1}\) is defined for every \(n
< \omega\).

Let \(\gamma = \sup_{n < \omega} \lambda_n\). We have 
\[\gamma^+ \leq \left|\prod_{n < \omega} \lambda_n\right| \leq \gamma^\omega =
2^\omega\cdot \gamma^+ = \gamma^+\] by SCH. (Since \(\lambda_0\) is measurable,
\(2^\omega < \lambda_0\).) Therefore \(U\) is \(\gamma^+\)-decomposable by
\cref{ProductDecomp}. It follows that \(U\) is \(\delta\)-decomposable for all
sufficiently large regular cardinals \(\delta < \gamma\). This contradicts that
\(U\) is not \(\delta_n\)-decomposable for any \(n < \omega\).
\end{proof}
\end{lma}
\end{comment}
\subsection{On the next almost strongly compact cardinal}
To discover the nontrivial relationships between compactness principles,
one must first dispense with the trivial ones. 
For example, any limit of strongly compact cardinals is almost
strongly compact. One is therefore led to ask whether every almost strongly compact
cardinal is either strongly compact or a limit of strongly compacts. This is
provable under the Ultrapower Axiom (\cite[Proposition 8.3.7]{UA}), but it is
conceivable that this hypothesis is unnecessary.

There is an easy characterization of precisely those almost strongly compact
cardinals that are strongly compact, essentially due to Menas, although he
proved it before the concept of an almost strongly compact cardinal had been
formulated:
\begin{thm}[Menas]\label{Menas} An almost strongly compact cardinal is strongly
compact if and only if it is measurable.\qed
\end{thm}
This theorem would seem to be optimal (since after all it is an equivalence). We
will show, however, that there are a priori weaker notions than measurability
that suffice to conclude that an almost strongly compact cardinal is strongly
compact.

\begin{prp}[SCH]\label{IfAlmostSC} Suppose \(\nu\) is a cardinal such that the least
\((\nu,\infty)\)-strongly compact cardinal \(\kappa\) is almost strongly
compact. Then \(\kappa\) is strongly compact.
\begin{proof}
By \cref{Menas}, it suffices to show that \(\kappa\) is measurable.
Let \(U\) be a \(\nu\)-complete uniform ultrafilter over \(\kappa^+\).
We claim that \(K_U\) is bounded below \(\kappa\). 
Assume \(K_U\) is unbounded, towards a contradiction.
By \cref{SimpleSpectrum} (using our SCH assumption), there
is some cardinal \(\eta < \kappa\) such that every regular cardinal \(\delta\)
with \(\eta\leq \delta < \kappa\) is in \(K_U\). 
In other words, \(U\) can be pushed forward to a uniform ultrafilter on \(\delta\)
for every regular cardinal \(\delta\)
with \(\eta\leq \delta < \kappa\). 
A pushforward of \(U\) is necessarily \(\nu\)-complete.
It follows that every regular cardinal greater than or equal to \(\eta\)
carries a \(\nu\)-complete uniform ultrafilter.
Therefore by Ketonen's Theorem
(\cref{KetonenThm}), \(\eta\) is \((\nu,\infty)\)-strongly compact, and this
contradicts the fact that \(\kappa\) is the least \((\nu,\infty)\)-strongly
compact cardinal.

Since \(K_U\) is bounded below \(\kappa\), we can now apply \cref{InternalDecomp} to
factor \(U\): thus \(j_U = (j_W)^{M_D}\circ j_D\) where \(D\) is an ultrafilter
on a cardinal \(\gamma < \kappa\) and \(W\in M_D\) is a uniform \(\kappa\)-complete
\(M_D\)-ultrafilter over \(j_D(\kappa^+)\). 

We will show that \(\kappa\) is regular. Given this, we can conclude the proposition by the
following argument. Since \(D\) lies on a cardinal less than \(\kappa\) and
\(\kappa\) is regular, \(j_D(\kappa) = \sup j_D[\kappa]\), and so 
every \(M_D\)-cardinal \(\delta < j_D(\kappa)\) has true
cardinality strictly less than \(\kappa\).
Therefore since \(W\) is \(\kappa\)-complete,
\(W\) is \(j_D(\kappa)\)-complete in \(M_D\). By elementarity,
\(\kappa^+\) carries a \(\kappa\)-complete uniform ultrafilter, as desired.

To finish, we show \(\kappa\) is regular. 
Suppose towards a contradiction that \(\kappa\) is singular, and
let \(\theta = \text{cf}(\kappa)\). Since \(M_D\) satisfies that \(W\) is
\(M_D\)-\(j_D(\theta^+)\)-complete and \(j_D(\kappa^+)\)-decomposable,
\((K_W)^{M_D}\) contains all sufficiently large \(M_D\)-regular cardinals below
\(j_D(\kappa)\) by \cref{KetonenSpectrum}. Therefore \(W\) witnesses that there
is a \((j_D(\nu),\infty)\)-strongly compact cardinal below \(j_D(\kappa)\) in
\(M_D\), contradicting the definition of \(\kappa\) (since \(j_D\) is
elementary).
\end{proof}
\end{prp}
At first glance, the proof appears to show that the least
\((\nu,\infty)\)-strongly compact cardinal is always strongly compact, but by a
theorem of \cite{MagidorBagaria}, one cannot prove (assuming ZFC + GCH) that the
least \((\omega_1,\infty)\)-strongly compact cardinal is {\it regular.} Where
does \cref{IfAlmostSC} use the almost strong compactness of \(\kappa\)? The
answer is that this hypothesis is required to apply \cref{InternalDecomp}.

\begin{thm}[SCH]\label{AlmostSCDichotomy} Suppose \(\kappa\) is an almost strongly
compact cardinal of uncountable cofinality. Then one of the following
holds:
\begin{itemize}
    \item \(\kappa\) is a strongly compact cardinal.
    \item \(\kappa\) is the successor of a strongly compact cardinal.
    \item \(\kappa\) is a limit of almost strongly compact cardinals.
\end{itemize}
\begin{proof}
    We may assume that \(\kappa\) is not a limit of almost strongly compact
    cardinals. We may also assume that \(\kappa\) is a limit cardinal. 
    
    Let \(\delta < \kappa\) be the supremum of the almost strongly
    compact cardinals below \(\kappa\).
    For each \(\alpha < \kappa\) with \(\alpha > \delta\), let \(f(\alpha)\) be
    the least cardinal \(\nu\) such that \(\alpha\) is not
    \((\nu,\infty)\)-strongly compact. The function \(f : \kappa\setminus
    \delta\to \kappa\) is regressive and nondecreasing, so since \(\kappa\) has
    uncountable cofinality, \(f\) assumes a constant value at all sufficiently
    large ordinals below \(\kappa\). In other words, there is a cardinal \(\nu <
    \kappa\) and an ordinal \(\alpha_0\) such that for all \(\alpha >
    \alpha_0\), \(f(\alpha) = \nu\). Thus \(\kappa\) is the least
    \((\nu,\infty)\)-strongly compact cardinal. By \cref{IfAlmostSC}, \(\kappa\)
    is strongly compact.
\end{proof}
\end{thm}

\begin{thm}\label{SuccessorSC} For any ordinal \(\alpha\),  if the
\((\alpha+1)\)-st almost strongly compact limit cardinal has uncountable cofinality,
it is  strongly compact.
\begin{proof}
Let \(\kappa\) be the \((\alpha+1)\)-st almost strongly compact limit
cardinal and assume that \(\kappa\) has uncountable cofinality.
Note that \(\kappa\) is not the least
almost strongly compact cardinal. We work in the collapse forcing extension \(N\) of \(V\)
in which the least strongly compact is countable. Notice that SCH holds in \(N\) 
since SCH holds above the least almost strongly compact cardinal in \(V\).
Therefore we can apply \cref{AlmostSCDichotomy} to conclude that \(\kappa\) is
strongly compact in \(N\). 
By L\'evy-Solovay, it follows that \(\kappa\) is strongly compact in \(V\).

If one wants to avoid forcing, one can just check that all the previous theorems go through
under the assumption that SCH holds at all sufficiently large cardinals below \(\kappa\).
\end{proof}
\end{thm}
\section{Cardinal preserving embeddings}\label{CardinalPreservingSection}
\subsection{Strong compactness and the Kunen inconsistency}
Kunen famously proved the inconsistency
of Reinhardt's ``ultimate large cardinal axiom''
asserting the existence of an elementary embedding from the universe of
sets to itself.
From a technical perspective,
the Kunen inconsistency places a bound on the degree of supercompactness
an elementary embedding \(j : V\to M\) can exhibit: there is always some 
\(\lambda < \kappa_\omega(j)\) such that \(M^\lambda \not \subseteq M\). Here
\(\kappa_\omega(j)\) is the supremum of the critical sequence of \(j\):
\begin{defn}
    Suppose \(j : M \to N\) is an elementary embedding between two transitive
    models of set theory. The {\it critical sequence of \(j\)} is the sequence
    \(\langle \kappa_n(j)\rangle_{n < \omega}\) defined by setting \(\kappa_n(j)
    = j^{(n)}(\text{crit}(j))\). The ordinal \(\kappa_\omega(j)\) is
    the supremum of the critical sequence of \(j\). 
\end{defn}
A natural (vague) question is whether there is a similar inconsistency theorem
for strong compactness, or in other words, a limitation on the covering
properties of inner models \(M\) such that there is an elementary embedding \(j
: V\to M\). 

For example, one might ask whether there can be an elementary embedding \(j :
V\to M\) where \(M\) is an inner model that has the \(\lambda\)-cover property
for all cardinals \(\lambda\); in other words, every \(A\subseteq M\) is
contained in some \(B\in M\) with \(|B| = |A|\).
The answer to this question, perhaps surprisingly, is yes. Suppose \(U\) is a
\(\kappa\)-complete ultrafilter over \(\kappa\). If \(2^\kappa > \kappa^+\),
then \(M_U\) does not have the \(\kappa^+\)-cover property, but assume instead
that the Generalized Continuum Hypothesis holds. Then \(M_U\) has the
\(\lambda\)-cover property for every cardinal \(\lambda\). To see this, it
suffices to see that for any cardinal \(\lambda\), \(j_U[\lambda]\) is covered
by a set \(A\in M_U\) with \(|A| = \lambda\). In fact, we can just take \(A =
\sup j_U[\lambda]\). 

A second question is whether there can be an elementary embedding \(j : V\to M\)
where \(M\) is an inner model with the {\it tight cover
property} at every cardinal: every \(A\subseteq M\) is contained in some \(B\in M\) with \(|B|^M = |A|\). 
(This is easily equivalent
to the question of whether there can be an embedding \(j : V\to M\) where \(M\)
is an inner model with the \(\lambda\)-cover property for all cardinals
\(\lambda\) as in the previous paragraph, with the additional requirement
that \(M\) and \(V\) have the same cardinals.) 
The answer here is an easy {\it no}. 

First note that \(M\) must be
closed under \(\omega\)-sequences. To see this,
fix a countable set \(\sigma\subseteq M\).
Let \(\tau\in M\) be an \(M\)-countable set containing \(\sigma\), and let
let \(f : \omega\to \tau\) be a surjection. 
Let \(x = f^{-1}[\sigma]\). Since \(j : V\to M\) is elementary and \(j(\omega) = \omega\),
\(x = j(x)\in M\). Since \(\sigma = f[x]\), \(\sigma\in M\).

We now reach a contradiction following Zapeltal's proof of the Kunen inconsistency.
Now let \(\lambda = \kappa_\omega(j)\).
Applying Shelah's Representation Theorem
\cite{AbrahamMagidor}, there exist regular cardinals \(\langle \delta_n\rangle_{n <
\omega}\) cofinal in \(\lambda\) for which there is a scale, or in other words an
increasing cofinal sequence \(\langle f_\alpha\rangle_{\alpha < \lambda^+}\) 
in the preorder \((\prod_{n < \omega} \delta_n,<^*)\).
Here \(f <^* g\) if \(f(n) < g(n)\) 
for all but finitely many \(n < \omega\).
Let \(\langle g_\alpha\rangle_{\alpha < \lambda^+} =
j(\langle f_\alpha\rangle_{\alpha < \lambda^+})\). 
Bu elementarity \(\langle g_\alpha\rangle_{\alpha < \lambda^+}\) is a scale 
for \(\langle j(\delta_n)\rangle_{n <
\omega}\) in \(M\), but since \(M\) is closed under countable sequences,
this is upwards absolute to \(V\).
Since \(j[\lambda^+]\) is cofinal in \(\lambda^+\),
\(\langle g_\alpha\rangle_{\alpha \in j[\lambda^+]}\)
is also a scale for \(\langle j(\delta_n)\rangle_{n <
\omega}\). Of course
\(g_{j(\alpha)} = j(f_\alpha)\), so
\(\langle j(f_\alpha)\rangle_{\alpha < \lambda^+}\)
is a scale for \(\langle j(\delta_n)\rangle_{n <
\omega}\).

Finally let \(h = \langle \sup j[\delta_n]\rangle_{n < \omega}\). 
We have \(\sup j[\delta_n] < j(\delta_n)\) since
\(j(\delta_n)\) is a regular cardinal larger than \(\delta_n\).
Therefore \(h\in \prod_{n < \omega} j(\delta_n)\).
But \(j(f_\alpha) < h\) for all \(\alpha < \lambda^+\),
\(g_\alpha = j\circ f_{j^{-1}(\alpha)}\) for any \(\alpha\in j[\lambda^+]\). 
This contradicts that \(\langle g_\alpha : \alpha \in
j[\lambda^+]\rangle\) is cofinal in \((\prod_{n <\omega} j(\delta_n),<^*)\).

\subsection{Strongly discontinuous embeddings}
The preceding proof
shows that the tight cover property 
is not really the right notion in this context.
A much more difficult question seems to be whether there can be an elementary embedding
\(j : V\to M\) such that 
for all cardinals \(\lambda\), 
\(j[\lambda]\) is contained in a set in \(M\) of \(M\)-cardinality \(\lambda\).
We call this the {\it local tight cover property}
because for any \(a\in M\) and any cardinal \(\lambda\),
\(j\) factors as \(V\stackrel{i}{\longrightarrow} N\stackrel{k}{longrightarrow}M\)
where \(N\) has the tight cover property at \(\lambda\) and \(a\in \text{ran}(k)\).

We will show that if there is a proper class of strongly compact
cardinals, then no such embedding can exist.
In fact, our proof rules out the weaker concept of a strongly
discontinuous embedding:
\begin{defn}
    Suppose \(j : V\to M\) is an elementary embedding. Then \(j\) is {\it
    strongly discontinuous} if for all cardinals \(\lambda\), if \(j(\lambda^+)
    \neq \lambda^+\) then \(j[\lambda^+]\) is bounded below \(j(\lambda^+)\).
\end{defn}
Obviously if \(\delta\) is regular, \(j(\delta) > \delta\), and 
\(j[\delta]\) is covered by a set \(C\) in \(M\)
of \(M\)-cardinality \(\delta\), then \(C\), and hence \(j[\delta]\), must be bounded
below the \(M\)-regular cardinal \(j(\delta)\).
Thus every elementary embedding with the local tight cover property is strongly discontinuous.

Strongly discontinuous embeddings also generalize the concept of a cardinal preserving embedding:
\begin{defn}
    Suppose \(M\) is an inner model. A nontrivial elementary embedding \(j : V\to M\) 
    is said to be {\it cardinal preserving} if \(\text{Card}^M = \text{Card}\).
\end{defn}
If \(j(\lambda^+)\) is a cardinal, then \(j(\lambda^+) =
j(\lambda)^+\). In particular, \(j(\lambda^+)\) is regular, so either
\(j[\lambda^+]\) is bounded below \(j(\lambda^+)\) or \(j(\lambda^+)  =
\lambda^+\). It follows that cardinal preserving embeddings are strongly
discontinuous.

\begin{prp}\label{AlmostSCMoves} Suppose \(j : V\to M\) is a strongly
discontinuous elementary embedding with critical point \(\kappa\) and
\(\delta\geq \kappa\) is an almost strongly compact cardinal. Then \(j(\delta) >
\delta\).
\begin{proof}
It suffices to prove this in the
case that \(\delta\) is an almost strongly compact limit cardinal.
Suppose towards a contradiction that \(j(\delta) = \delta\). Note that for all
\(\alpha < \kappa\), \(j(\delta^\alpha) = (\delta^{+\alpha})^M \leq
\delta^{+\alpha}\), and so \(j(\delta^{+\alpha}) = \delta^{+\alpha}\). It
follows that \((\delta^{+\kappa})^M = \delta^{+\kappa}\). In fact,
\((\delta^{+\kappa+1})^M = \delta^{+\kappa+1}\) by a standard argument. (For any
wellorder \(\preceq\) of \(\delta^{+\kappa}\), \(j({\preceq})\cap
\delta^{+\kappa}\) belongs to \(M\) and has length at least
\(\text{ot}(\preceq)\). Thus \((\delta^{+\kappa+1})^M > \text{ot}(\preceq).\))

On the other hand, \(j(\delta^{+\kappa}) = (\delta^{+j(\kappa)})^{M} >
(\delta^{+\kappa+1})^M = \delta^{+\kappa+1}\). Since \(j\) is strongly
discontinuous, we must therefore have \(j(\delta^{+\kappa+1}) > \sup
j[\delta^{+\kappa+1}]\).

Let \(U\) be the ultrafilter over \(\delta^{+\kappa+1}\) derived from \(j\)
using \(\sup j[\delta^{+\kappa+1}]\). Then for all \(\alpha < \kappa\),
\(j_U(\delta^{+\alpha}) = \delta^{+\alpha}\), so \(U\) is
\(\delta^{+\alpha}\)-indecomposable.

Since \(\delta\) is almost strongly compact, \(j_U = (j_W)^{M_D}\circ j_D\)
where \(D\) is an ultrafilter over a cardinal \(\eta < \delta\) and \(W\in M_D\)
is an \(M_D\)-ultrafilter over \(j_D(\delta^{+\kappa+1})\) that is
\(M_D\)-\(j_D(\gamma)\)-complete in \(M_D\) for all \(\gamma <
\delta^{+\kappa}\). Working in \(M_D\), let \(\zeta = \text{crit}(W)\). Then
\(\delta < \zeta \leq (\delta^{+j_D(\kappa)+1})^{M_D}\). This contradicts that
\(\zeta\) is measurable and therefore inaccessible.
\end{proof}
\end{prp}

\begin{thm}If there is a proper class of almost
strongly compact cardinals, then there are no strongly discontinuous embeddings.
\begin{proof}
If \(j : V\to M\) is an elementary embedding, then \(j\) must fix an almost
strongly compact cardinal above its critical point \(\kappa\) since \(j\) is
continuous at ordinals of cofinality \(\omega\) and the class of almost strongly
compact cardinals is closed. (Let \(\delta_0\) be the least almost strongly
compact cardinal, and for \(n < \omega\), let \(\delta_{n+1}\) be the least
almost strongly compact cardinal above \(j(\delta_n)\). Then \( \sup_{n<\omega}
\delta_n\) is an almost strongly compact cardinal that is fixed by \(j\).)
Therefore \(j\) is not strongly discontinuous by \cref{AlmostSCMoves}.
\end{proof}
\end{thm}

Given our observations above, the following is an immediate corollary:
\begin{thm}\label{NoStronglyDiscont} If there is a proper class of almost
    strongly compact cardinals, then there are no cardinal preserving embeddings.\qed
\end{thm}

We note the following fact, which improves on an observation due to Caicedo:
\begin{prp}
Suppose \(j : V\to M\) is a strongly discontinuous embedding with critical point
\(\kappa\). Then \(\kappa\) is \(\lambda\)-strongly compact for every \(\lambda
< \kappa_\omega(j)\).
\begin{proof}
Suppose \(\delta\) is a regular cardinal such that \(\kappa\leq \delta <
\kappa_{\omega}(j)\). Then \(\delta^+\) carries a uniform \(\kappa\)-complete
ultrafilter \(U\), namely the ultrafilter derived from \(j\) using \(\sup
j[\delta^+]\). Since \(\delta\) is regular, then \(U\) is necessarily
\(\delta\)-decomposable by a theorem of Kunen-Prikry \cite{PrikryKunen}. In
particular, \(\delta\) carries a uniform \(\kappa\)-complete ultrafilter.
Applying Ketonen's Theorem (\cref{KetonenThm}), \(\kappa\) is
\(\lambda\)-strongly compact for all \(\lambda < \kappa_\omega(j)\).
\end{proof}
\end{prp}
On the other hand, \(\kappa_\omega(j)\) cannot be a limit of
\(\kappa_\omega(j)^{+\kappa}\)-strongly compact cardinals by the proof of
\cref{NoStronglyDiscont}.

\section{Definability and ultrafilters}\label{HODSection}
The results of this section are a ZFC analog of the following theorem:
\begin{thm}\label{UAGAThm}
    Assume the Ground Axiom,\footnote{
        The {\it Ground Axiom} asserts that \(V\) is not a set generic extension of
        any inner model \(M\subsetneq V\).
    }
    the Ultrapower Axiom, and the existence of a
    strongly compact cardinal. Then every set is definable from an ordinal.\qed
\end{thm}

We will prove the following generalization:
\begin{thm}\label{GAThm} Assume the Ground Axiom. Then for any strongly compact
cardinal \(\kappa\), every set is definable from a \(\kappa\)-complete
ultrafilter over an ordinal.\qed
\end{thm}
Under UA, every countable complete ultrafilter over an ordinal is ordinal definable,
so \cref{GAThm} implies \cref{UAGAThm}.

The proof (which appears below \cref{Ground}) involves the following collection
of structures:
\begin{defn}
    Let \(\kappa\text{-OD}\) denote the class of sets that are definable from a
    \(\kappa\)-complete ultrafilter over an ordinal, and let \(\kappa\text{-HOD}\) denote the
    class of hereditarily \(\kappa\)-OD sets.
\end{defn}
Note that \(x\) is \(\kappa\)-OD if and only if \(x\) is in
\(\textnormal{OD}_U\) for some \(\kappa\)-complete ultrafilter \(U\) over an
ordinal, so \(\kappa\)-OD is first-order definable, and therefore so is \(\kappa\)-HOD. 

The following basic observation sets things in motion:
\begin{thm}\label{kappaHODZFC} For any cardinal \(\kappa\), the class
\(\kappa\textnormal{-HOD}\) is an inner model of \(\textnormal{ZF}\). If
\(\kappa\) is strongly compact, then \(\kappa\textnormal{-HOD}\) satisfies the
Axiom of Choice.
\begin{proof}
The proof that \(\kappa\textnormal{-HOD}\) satisfies ZF is just like the usual
proof that \(\text{HOD}\) satisfies ZF, so we omit it. The issue in showing that 
\(\kappa\textnormal{-HOD}\) satisfies the Axiom of
Choice is that the class of \(\kappa\)-complete ultrafilters over ordinals is not
naturally wellordered.

Assume that \(\kappa\) is strongly compact. 
The key idea is that in this case, for any ordinal \(\delta\), 
there is actually a \(\kappa\)-OD
wellorder of the \(\kappa\)-complete ultrafilters over \(\delta\). 
Let \(\mathcal W\) be a \(\kappa\)-complete fine ultrafilter
over \(P_\kappa(P(\delta))\). 
Let \(W\) be an ultrafilter over an ordinal such that \(\mathcal W\) and \(W\) are
Rudin-Keisler equivalent,
or in other words, \(j_{\mathcal W} = j_W\).
For each \(\kappa\)-complete ultrafilter \(U\) over \(\delta\), 
let \(\alpha_U\) be the least ordinal such that
\(U\) is the ultrafilter on \(\delta\) derived from \(j_W\) using \(\alpha_U\).\footnote{
    It is a standard fact that \(\alpha_U\) exists, so we include the proof
    in fine print. Let \(\sigma = [\text{id}]_\mathcal W\).
    Since \(\mathcal W\) is a fine ultrafilter over \(P_\kappa(P(\delta))\),
    \(|\sigma |^{M_{\mathcal W}} < j_\mathcal W(\kappa)\) and \(j_\mathcal W[P(\delta)]\subseteq \sigma\).
    Therefore \(j_\mathcal W[U]\subseteq \sigma\cap j_\mathcal W(U)\).
    Since \(j_\mathcal W(U)\) is \(j_\mathcal W(\kappa)\)-complete,
    \(\bigcap (\sigma\cap j_\mathcal W(U))\) is nonempty. 
    It follows that there is some ordinal \(\alpha\in \bigcap j_\mathcal W[U]\).
    Clearly \(U\) is the ultrafilter over \(\delta\) derived from \(j_W\) using \(\alpha\).
}

The function \(U\mapsto \alpha_U\) is injective and 
\(\kappa\)-OD. It follows that there is a \(\kappa\)-OD wellorder
of the set of \(\kappa\)-complete ultrafilters
over \(\delta\).

Now that one has a \(\kappa\)-OD wellorder of the \(\kappa\)-complete
ultrafilters over \(\delta\) for each \(\delta\), it is easy to construct a
\(\kappa\)-OD wellorder of the sets that are ordinal definable from a
\(\kappa\)-complete ultrafilter over \(\delta\). This suffices to show that
\(\kappa\text{-HOD}\) satisfies the Axiom of Choice.
The proof is the same as the proof that AC holds in \(\text{HOD}\).
\end{proof}
\end{thm}

Note that the Axiom of Choice also holds \(\omega\)-HOD,
since in fact \(\omega\)-HOD is \(V\):
\begin{prp}\label{OmegaThm} \(V = \omega\textnormal{-HOD}\).
    \begin{proof}
    Let \(M = \omega\text{-HOD}\).
    By the definition of \(M\), \(\omega\)-complete ultrafilters descend to \(M\):
    in fact, if \(U\) is an \(M\)-ultrafilter, then \(U\in M\),
    since \(U\) extends to an ultrafilter \(W\) which is isomorphic to an 
    ultrafilter \(Z\) over an ordinal; \(j_Z\) is \(\omega\)-OD,
    so \(j_Z\restriction M\) is close to \(M\), and so since \(U\) is a
    derived \(M\)-ultrafilter of \(j_Z\restriction M\), \(U\in M\).

    Since \(M\) is closed under finite sequences,
    \(M\) has the \(\omega\)-cover property.
    Although we have not shown that \(\omega\)-HOD
    satisfies the Axiom of Choice,
    the proof of \cref{AppxEquiv} still goes through
    with \(\kappa = \omega\).
    It follows that \(M\) has the \(\omega\)-approximation property,
    which of course implies that \(V = M\).
    \end{proof}
\end{prp}

We will need the analog of Vopenka's Theorem for \(\kappa\text{-HOD}\). (The proof
requires no real modification.)
\begin{lma}\label{Vopenka} For any strongly compact cardinal \(\kappa\), for any
set of ordinals \(A\), \(\kappa\textnormal{-HOD}_A\) is a set-generic extension
of \(\kappa\textnormal{-HOD}\).
\begin{proof}
We first note that \cref{kappaHODZFC} relativizes to show that
\(\kappa\textnormal{-HOD}_A\) is an inner model of ZFC.
To show that \(\kappa\textnormal{-HOD}_A\) is a set-generic extension
of \(\kappa\textnormal{-HOD}\),
it therefore suffices to verify Bukovsky's criterion \cite{Usuba} by showing that
\(\kappa\textnormal{-HOD}\) has the \((2^\rho)^+\)-uniform cover property in 
\(\kappa\textnormal{-HOD}_A\) where \(\rho = \sup A\).\footnote{
    To apply Bukovsky's Theorem,
    it is essential that \(\kappa\textnormal{-HOD}\)
    is a model of AC; this is our only signifiant use of \cref{kappaHODZFC}.
} 

Let \(\gamma\) be an ordinal and let \(f : \gamma \to
\gamma\) be a function that is \(\kappa\)-OD from \(A\). The function
 \(g : \gamma\times P(\rho)\to \gamma\) defined by
\(f(\alpha) = g(\alpha,A)\) for all \(\alpha < \gamma\) is then
\(\kappa\)-OD. Let \(F(\alpha) =
\{g(\alpha,B): B\subseteq \gamma\}\). Then \(F\) is \(\kappa\)-OD, \(|F(\alpha)|
< (2^\rho)^+\) for all \(\alpha < \gamma\), and \(f(\alpha) \in F(\alpha)\) for
all \(\alpha < \gamma\). This verifies that \(\kappa\textnormal{-HOD}_A\) has
the \((2^\rho)^+\)-uniform covering property.
\end{proof}
\end{lma}

\begin{prp}\label{HODA} Suppose \(\kappa\) is strongly compact and \(A\) is a
set of ordinals such that \(V_\kappa\subseteq \kappa\textnormal{-HOD}_A\). Then
\(V = \kappa\textnormal{-HOD}_A\).
\begin{proof}
Let \(N\) denote \(\kappa\textnormal{-HOD}_A\). We first show that \(N\) is
closed under \(\kappa\)-sequences.
To show
\(N\) is closed under \(\kappa\)-sequences, it therefore suffices to show that
for all ordinals \(\lambda\), \({}^\kappa\lambda\subseteq N\). Let \(U\) be a
\(\kappa\)-complete ultrafilter over an ordinal \(\nu\) such that \(j_U(\kappa)
> \lambda\). Then 
\[{}^\kappa\lambda\subseteq j_U(V_\kappa)\subseteq j_U(N)\subseteq N\]

We justify this last inclusion. It suffices to show that every set that \(M_U\)
thinks is definable from a \(j_U(\kappa)\)-complete ultrafilter over an
ordinal is truly definable (in \(V\)) from a \(\kappa\)-complete ultrafilter
over an ordinal. Since \(M_U\) is definable from the \(\kappa\)-complete
ultrafilter \(U\), it therefore suffices to show that every
\(j_U(\kappa)\)-complete ultrafilter \(W\) of \(M_U\) over an ordinal \(\gamma\)
is definable from a \(\kappa\)-complete ultrafilter over an ordinal. But
consider the following \(\kappa\)-complete ultrafilter:
\[U\text{-}\textstyle\sum W = \{A\subseteq \nu\times \bar \gamma : [\alpha \mapsto A_\alpha]_U\in W\}\] where \(\bar \gamma\) is the least
ordinal such that \(j_U(\bar \gamma) \geq \gamma\). We have
\[W = \{[f]_U : \text{ran}(f)\subseteq P(\bar \gamma) \text{ and
}\{(\alpha,\beta) : \beta\in f(\alpha)\}\in U\text{-}\textstyle\sum W\}\] so \(W\) is definable from
\(U\text{-}\textstyle\sum W\).

We now show that \(N\) has the \(\kappa\)-approximation and cover properties.
Since \(N\) is closed under \({<}\kappa\)-sequences, \(N\) certainly has the
\(\kappa\)-cover property. Since \(N\) satisfies the Axiom of Choice and \(U\cap
N\in N\) for any \(\kappa\)-complete ultrafilter over an ordinal, it easily
follows that \(U\cap N\in N\) for any ultrafilter \(U\) over a set that belongs
to \(N\). Therefore by \cref{AppxEquiv}, \(N\) has the \(\kappa\)-approximation
property. 

It follows that \(V = N\), since \(V\) is the unique inner model with
the \(\kappa\)-approximation and cover properties that contains
\(H_{\kappa^+}\). (This last fact is a consequence of the proof of the
definability of inner models with the \(\kappa\)-approximation and cover
properties.)
\end{proof}
\end{prp}

\begin{thm}\label{Ground} For any strongly compact cardinal \(\kappa\), \(V\) is
a set generic extension of \(\kappa\textnormal{-HOD}\).
\begin{proof}
Let \(A\) be a set of ordinals such that \(V_\kappa\subseteq
\kappa\textnormal{-HOD}_A\); for example, \(A\) can be chosen to code a
wellfounded extensional relation \(R\subseteq \kappa\times \kappa\) whose
transitive collapse is \(V_\kappa\). By \cref{HODA}, \(V =
\kappa\text{-HOD}_A\), and by \cref{Vopenka}, \(\kappa\text{-HOD}_A\) is a
generic extension of \(\kappa\text{-HOD}\), so \(V\) is a generic extension of
\(\kappa\text{-HOD}\).
\end{proof}
\end{thm}
\cref{Ground} of course immediately implies \cref{GAThm}. 

Given \cref{OmegaThm},
it is natural to speculate that \cref{Ground} is just a precursor to 
a proof that \(V = \kappa\text{-HOD}\) for all strongly compact cardinals \(\kappa\).
But of course this is not the case:
\begin{prp}\label{Cohen}
It is consistent with \textnormal{ZFC} that 
there is a strongly compact cardinal \(\kappa\) such that 
\(V \neq \kappa\textnormal{-HOD}\).
\begin{proof}
Assume there is a strongly compact cardinal. Let \(g\) be \(V\)-generic for
Cohen forcing. Note that any \(\kappa\)-complete ultrafilter \(U\) of \(V[g]\)
over an ordinal is definable over \(V[g]\) from a \(\kappa\)-complete
ultrafilter in \(V\): indeed, \(U\) is the unique ultrafilter extending \(U\cap
V\), and \(U\cap V\in V\), by the L\'evy-Solovay Theorem \cite{LevySolovay}. It
follows that \[\kappa\textnormal{-HOD}^{V[g]}\subseteq
\kappa\textnormal{-HOD}^{V}\] by the homogeneity of Cohen forcing. Therefore
\(g\notin \kappa\textnormal{-HOD}^{V[g]}\).
\end{proof}
\end{prp}

Under large cardinal hypotheses, the Ground Axiom is equivalent to the
statement that every set is ordinal definable from a countably complete ultrafilter
on an ordinal in all generic extensions:
\begin{thm}\label{GroundEqui}
    Assume there is a proper class of strongly compact cardinals.
    Then the following are equivalent:
    \begin{enumerate}[(1)]
        \item The Ground Axiom.
        \item \(V\) is the intersection of all
             models of the form \((\delta\textnormal{-HOD})^{V^{\mathbb B}}\)
             where \(\delta\) is a cardinal and \(\mathbb B\) is a complete Boolean algebra.
        \item In any generic extension \(N\), every set in \(V\) is ordinal definable 
             using an internal ultrapower embedding of \(N\) as a predicate.
        \item Every generic extension \(N\) satisfies that every set in \(V\) is ordinal definable 
             using an elementary embedding from \(N\) into an inner model that is closed 
             under \((\omega_1)^N\)-sequences.
    \end{enumerate}
    \begin{proof}
        {\it (1) implies (2):} Fix a set \(x\), a cardinal \(\delta\),
        and a complete Boolean algebra \(\mathbb B\). We must show that \(x\)
        is \(\delta\)-OD in \(V^{\mathbb B}\). Let \(\kappa > \delta\cdot |\mathbb B|\) 
        be a strongly compact cardinal. Then \(x\) is \(\kappa\)-OD in \(V\) 
        by \cref{GAThm}. But since the Ground Axiom holds, \(V\) is definable over
        \(V^{\mathbb B}\) without parameters (as the intersection of all
        grounds of \(V^{\mathbb B}\). Moreover since \(\kappa > |\mathbb B|\),
        every \(\kappa\)-complete ultrafilter of \(V\) is ordinal definable in
        \(V^{\mathbb B}\) from its unique extension to \(V^{\mathbb B}\).
        It follows that \(x\) is \(\kappa\)-OD in \(V^{\mathbb B}\).
        Since \(\kappa \geq \delta\), this implies (2).

        {\it (2) implies (3):} This is easy given the observation that 
        a countably complete ultrafilter over an ordinal
        is ordinal definable from its associated ultrapower embedding.

        {\it (3) implies (4):} Trivial.

        {\it (4) implies (1):} Fix a ground \(M\) of \(V\). We must show \(M = V\).
        Let \(\mathbb B\) be a complete Boolean algebra of \(M\) such that \(V = M[G]\) for
        some \(M\)-generic ultrafilter \(G\) on \(\mathbb B\).
        Let \(\delta > |\mathbb B|\) and let \(H\subseteq \text{Col}(\omega,\delta)\) be
        a \(V\)-generic filter. By the universality of collapse forcing,
        there is an \(M\)-generic filter \(F\subseteq \text{Col}(\omega,\delta)\) 
        in \(M[G][H]\) such that \(M[F] = M[G][H]\). 
        
        Fix a set of ordinals \(x\in V\), and we will show that \(x\in M\). 
        Now let \(N = M[G][H]\) and fix in \(N\) an elementary embedding \(j : N\to P\)
        such that \(P^{{(\omega_1)^N}}\cap N\subseteq P\) and 
        \(x\) is ordinal definable using \(j\) as a predicate.
        By the Reflection Theorem,\footnote{Formally, we are working in 
        von Neumann-Bernays-G\"odel class theory (NBG). For any formula \(\varphi\) in the language of
        first-order set theory with an additional predicate symbol, NBG proves that
        for all classes \(A\), there is an ordinal \(\alpha\) such that for all \(x\in V_\alpha\),
        \((V_\alpha,A\cap V_\alpha)\) satisfies \(\varphi(x)\) if and only if \((V,A)\) satisfies \(\varphi(x)\).} there is some ordinal \(\alpha\) such that
        \(x\) is ordinal definable in \(N_\alpha\) using \(j\restriction N_\alpha\)
        as a predicate. (Here \(N_\alpha = N\cap V_\alpha\).)
        
        Since \(M\) is an
        \(\omega_1^N\)-pseudoground of \(N\) and \(P^{(\omega_1)^N}\cap N\subseteq P\), 
        the Hamkins Universality Theorem 
        (\cref{HamkinsUniversality}) implies that \(j\restriction M_\alpha\) belongs to \(M\).
        Note however that \(j\restriction N_\alpha\) is definable in \(N\) from 
        \(j\restriction M_\alpha\) and \(M\). It follows that \(x\) is definable in
        \(N\) from the predicate for \(M\) and parameters in \(M\).
        Since \(N = M[F]\) where \(F\subseteq \text{Col}(\omega,\delta)\)
        is \(M\)-generic, the homogeneity of collapse forcing implies that 
        \(x\cap M\) is definable over \(M\). 
        Since \(x\) is a set of ordinals, \(x\cap M = x\),
        so \(x\in M\).

        It follows that every set of ordinals belongs to \(M\), which proves \(V = M\).
        Since \(M\) was an arbitrary ground, the Ground Axiom holds.
    \end{proof}
\end{thm}

Suppose \(\kappa\) is strongly compact. We do not know whether \(\kappa\) must be
strongly compact in \(\kappa\text{-HOD}\).
We can, however, prove the following ``cheap HOD Conjecture'':
\begin{thm}\label{CheapHOD} Suppose \(\kappa\) is supercompact. Then \(\kappa\)
is supercompact in \(\kappa\textnormal{-HOD}\), and in fact
\(\kappa\textnormal{-HOD}\) is a weak extender model for the
supercompactness of \(\kappa\).
\begin{proof}
Let \(N\) denote \(\kappa\)-HOD. By \cref{Ground}, \(N\) is a ground of \(V\),
so there is some cardinal \(\lambda\) such that for all regular cardinals
\(\delta\geq \lambda\), any set \(S\) that is stationary in \(N\) is stationary
in \(V\).

Fix \(\delta\geq \lambda\). We will show that for any normal fine
\(\kappa\)-complete ultrafilter \(\mathcal U\) over \(P_\kappa(\delta)\),
\(\mathcal U\cap N\in N\) and \(P_\kappa(\delta)\cap N\in \mathcal U\). 
This establishes that \(N\) is a weak extender model for the
supercompactness of \(\kappa\).

Of course, \(\mathcal U\cap N\in N\) by the definition of \(N\) for any
\(\kappa\)-complete normal fine ultrafilter \(\mathcal U\) over \(P_\kappa(\delta)\);
this is because \(\mathcal U\) is the unique
normal fine ultrafilter on \(P_\kappa(\delta)\)
that is Rudin-Keisler equivalent to \(\mathcal U\),
and hence \(\mathcal U\) is ordinal definable from \(j_\mathcal U\).

We now show that \(P_\kappa(\delta)\cap N\in \mathcal U\). Let \(j : V \to M\)
be the ultrapower embedding associated to \(\mathcal U\). By \L o\'s's Theorem, and
since \([\text{id}]_\mathcal U = j[\delta]\), we just need to show that
\(j[\delta]\in j(N)\). Let \(T\) be the set of ordinals less than \(\delta\)
that have cofinality \(\omega\) in \(N\). Let \(\langle S_\alpha\rangle_{\alpha
< \delta}\in N\) be a partition of \(T\) into stationary sets. Let \(\langle
S_\alpha^*\rangle_{\alpha < j(\delta)} = j(\langle S_\alpha\rangle_{\alpha <
\delta})\). Thus \(\langle S_\alpha^*\rangle_{\alpha < j(\delta)}\in j(N)\).
Then \(j[\delta] = \{\alpha < j(\delta) : S_\alpha^*\cap \sup j[\delta]\text{ is
stationary}\}\) by the proof of Solovay's Theorem; see \cite[Corollary
4.4.31]{UA}. Thus \(j[\delta]\) is ordinal definable in \(M\) from \(\langle
S_\alpha^*\rangle_{\alpha < j(\delta)}\), so \(j[\delta]\in j(N)\).
\end{proof}
\end{thm}

\section{Questions}
\begin{qst}[Boney-Unger and Brooke-Taylor]
If there is a proper class of almost strongly compact cardinals, is there 
a proper class of strongly compact cardinals?
\end{qst}

\begin{qst}[Caicedo]
Can there be cardinal preserving or cofinality preserving embeddings from
the universe of sets into an inner model?
\end{qst}

\begin{qst}
    Suppose \(\kappa\) is strongly compact. Is \(\kappa\) strongly compact in
    \(\kappa\text{-HOD}\)?
\end{qst}
\bibliography{Bibliography.bib}

\begin{thebibliography}{10}

\bibitem{UA}
Gabriel Goldberg.
\newblock {\em The {U}ltrapower {A}xiom}.
\newblock PhD thesis, Harvard University, 2019.

\bibitem{Woodin}
W.~Hugh Woodin.
\newblock In search of {U}ltimate-{$L$}: the 19th {M}idrasha {M}athematicae
  {L}ectures.
\newblock {\em Bull. Symb. Log.}, 23(1):1--109, 2017.

\bibitem{Silver}
Jack~H. Silver.
\newblock Indecomposable ultrafilters and {$0^\#$}.
\newblock In {\em Proceedings of the {T}arski {S}ymposium ({P}roc. {S}ympos.
  {P}ure {M}ath., {V}ol. {XXV}, {U}niv. {C}alif., {B}erkeley, {C}alif., 1971)},
  pages 357--363, 1974.

\bibitem{Boney}
Will Boney and Spencer Unger.
\newblock Large cardinal axioms from tameness in {AEC}s.
\newblock {\em Proc. Amer. Math. Soc.}, 145(10):4517--4532, 2017.

\bibitem{Caicedo}
Andr\'es~Eduardo Caicedo.
\newblock Cardinal preserving elementary embeddings.
\newblock In {\em Logic {C}olloquium 2007}, volume~35 of {\em Lect. Notes
  Log.}, pages 14--31. Assoc. Symbol. Logic, La Jolla, CA, 2010.

\bibitem{Cohen}
Paul Cohen.
\newblock The independence of the continuum hypothesis.
\newblock {\em Proc. Nat. Acad. Sci. U.S.A.}, 50:1143--1148, 1963.

\bibitem{Godel}
K.~G\"odel.
\newblock The consistency of the axiom of choice and of the generalized
  continuum-hypothesis.
\newblock {\em Proceedings of the National Academy of Sciences of the United
  States of America}, page 556?557, 1938.

\bibitem{KunenLU}
Kenneth Kunen.
\newblock Some applications of iterated ultrapowers in set theory.
\newblock {\em Ann. Math. Logic}, 1:179--227, 1970.

\bibitem{MitchellSteel}
William~J. Mitchell and John~R. Steel.
\newblock {\em Fine structure and iteration trees}, volume~3 of {\em Lecture
  Notes in Logic}.
\newblock Springer-Verlag, Berlin, 1994.

\bibitem{NeemanSteel}
Itay Neeman and John Steel.
\newblock Equiconsistencies at subcompact cardinals.
\newblock {\em Arch. Math. Logic}, 55(1-2):207--238, 2016.

\bibitem{FiniteLevels}
W.~Hugh Woodin.
\newblock {\em Fine structure at the finite levels of supercompactness}.
\newblock To appear.

\bibitem{Sheard}
Michael Sheard.
\newblock Indecomposable ultrafilters over small large cardinals.
\newblock {\em J. Symbolic Logic}, 48(4):1000--1007 (1984), 1983.

\bibitem{MagidorBagaria}
Joan Bagaria and Menachem Magidor.
\newblock On {$\omega_1$}-strongly compact cardinals.
\newblock {\em J. Symb. Log.}, 79(1):266--278, 2014.

\bibitem{Kunen}
Kenneth Kunen.
\newblock Elementary embeddings and infinitary combinatorics.
\newblock {\em J. Symbolic Logic}, 36:407--413, 1971.

\bibitem{HamkinsApprox}
Joel~David Hamkins.
\newblock Extensions with the approximation and cover properties have no new
  large cardinals.
\newblock {\em Fund. Math.}, 180(3):257--277, 2003.

\bibitem{Tarski}
Alfred Tarski.
\newblock Some problems and results relevant to the foundations of set theory.
\newblock In {\em Logic, {M}ethodology and {P}hilosophy of {S}cience ({P}roc.
  1960 {I}nternat. {C}ongr.)}, pages 125--135. Stanford Univ. Press, Stanford,
  Calif., 1962.

\bibitem{Hamkins}
Joel~D. Hamkins.
\newblock Tall cardinals.
\newblock {\em MLQ Math. Log. Q.}, 55(1):68--86, 2009.

\bibitem{Magidor}
Menachem Magidor.
\newblock How large is the first strongly compact cardinal? or {A} study on
  identity crises.
\newblock {\em Ann. Math. Logic}, 10(1):33--57, 1976.

\bibitem{MitchellSkies}
William Mitchell.
\newblock Indiscernibles, skies, and ideals.
\newblock In {\em Axiomatic set theory ({B}oulder, {C}olo., 1983)}, volume~31
  of {\em Contemp. Math.}, pages 161--182. Amer. Math. Soc., Providence, RI,
  1984.

\bibitem{Viale}
Matteo Viale.
\newblock A family of covering properties.
\newblock {\em Math. Res. Lett.}, 15(2):221--238, 2008.

\bibitem{Ketonen}
Jussi Ketonen.
\newblock Strong compactness and other cardinal sins.
\newblock {\em Ann. Math. Logic}, 5:47--76, 1972/73.

\bibitem{AbrahamMagidor}
Uri Abraham and Menachem Magidor.
\newblock Cardinal arithmetic.
\newblock In {\em Handbook of {S}et {T}heory. {V}ols. 1, 2, 3}, pages
  1149--1227. Springer, Dordrecht, 2010.

\bibitem{PrikryKunen}
Kenneth Kunen and Karel Prikry.
\newblock On descendingly incomplete ultrafilters.
\newblock {\em J. Symbolic Logic}, 36:650--652, 1971.

\bibitem{Usuba}
Toshimichi Usuba.
\newblock The downward directed grounds hypothesis and very large cardinals.
\newblock {\em J. Math. Log.}, 17(2):1750009, 24, 2017.

\bibitem{LevySolovay}
A.~L\'{e}vy and R.~M. Solovay.
\newblock Measurable cardinals and the continuum hypothesis.
\newblock {\em Israel J. Math.}, 5:234--248, 1967.

\end{thebibliography}
\bibliographystyle{unsrt}
\end{document}